\documentclass{article}

\newif\ifaistats
\newif\ifarxiv
\newif\ifneurips
\newif\ificml
\newif\ifaistatsnot
\newif\ifarxivnot
\newif\ifneuripsnot
\newif\ificmlnot

\newif\ifextension
\newif\ifextensionnot
\extensionfalse
\extensionnottrue
\newif\ifhidden
\newif\ifhiddennot
\hiddenfalse
\hiddennottrue

\aistatsfalse
\arxivfalse
\neuripsfalse
\icmlfalse
\aistatsnottrue
\arxivnottrue
\neuripsnottrue
\icmlnottrue

\ifaistats \aistatsnotfalse \fi
\ifarxiv \arxivnotfalse \fi
\ifneurips \neuripsnotfalse \fi
\ificml \icmlnotfalse \fi
\ifextension \extensionnotfalse \fi
\ifhidden \hiddennotfalse \fi
\arxivtrue
\ifaistats \aistatsnotfalse \fi
\ifarxiv \arxivnotfalse \fi
\ifneurips \neuripsnotfalse \fi
\ificml \icmlnotfalse \fi
\ifextension \extensionnotfalse \fi
\ifhidden \hiddennotfalse \fi

\usepackage[utf8]{inputenc} 
\usepackage[T1]{fontenc}    
\usepackage{hyperref}       
\usepackage{url}            
\usepackage{booktabs}       
\usepackage{amsfonts}       
\usepackage{nicefrac}       
\usepackage{microtype}      
\usepackage{xcolor}         
\usepackage{wrapfig}
\ifaistatsnot
    \usepackage{natbib}
\fi
\usepackage{xargs}                      
\usepackage{comment}

\usepackage{hyperref}
\usepackage{url, times}

\usepackage{amssymb,amsmath,amscd,amsfonts,amstext,amsthm,bbm, enumerate, dsfont, mathtools}
\usepackage{array}
\usepackage{multicol}
\usepackage{graphicx}
\usepackage{varwidth}
\usepackage{hyperref}
\usepackage{import}
\usepackage{caption,subcaption}
\ificmlnot
    \usepackage{algorithm}
    \usepackage{algpseudocode}
    \usepackage{cleveref}
\fi
\ificml
    \usepackage[capitalize,noabbrev]{cleveref}
\fi
\usepackage[none]{hyphenat}
\usepackage{footnote}
\usepackage{makecell}
\usepackage{threeparttable}
\usepackage{booktabs}
\usepackage{tcolorbox}
\usepackage{caption}
\usepackage{subcaption}

\usepackage{scrextend}

\usepackage{tablefootnote}
\usepackage{changepage}
\usepackage{stackengine}
\usepackage{pifont}

\usepackage{cleveref}

\usepackage{mdframed}

\graphicspath{ {images/} }

\newtheorem{theorem}{Theorem}
\newtheorem{proposition}{Proposition}
\newtheorem{lemma}{Lemma}

\newtheorem{assumption}{Assumption}
\newtheorem*{remark}{Remark}
\newtheorem{definition}{Definition}

\DeclareMathOperator*{\argmin}{argmin}

\newcommand{\R}{\mathbb{R}}

\newcommand{\rdtor}{\R^d \to \R}

\newcommand{\eqdef}{\stackrel{\text{def}}{=}}
\newcommand{\cO}{{\cal O}}
\newcommand{\cOb}[1]{\cO \lr #1 \rr}

\mdfdefinestyle{definition}{
	backgroundcolor=lightgray!20,
    hidealllines=true,
}
\mdfdefinestyle{theorem}{
    linecolor=white,
	backgroundcolor=blue!3,
}
\mdfdefinestyle{proof}{
    linecolor=orange,
}
\mdfdefinestyle{proposition}{
    backgroundcolor=orange!5,
    hidealllines=true,
}
\mdfdefinestyle{corollary}{
    linecolor=blue!50,
}
\mdfdefinestyle{lemma}{
    linecolor=blue!50,
}

\surroundwithmdframed[style=lemma]{theorem}
\surroundwithmdframed[style=corollary]{corollary}
\surroundwithmdframed[style=lemma]{lemma}

\newcolumntype{?}{!{\vrule width 1pt}}
\usepackage{colortbl}

\definecolor{mydarkgreen}{RGB}{39,130,67}

\definecolor{mydarkred}{RGB}{192,47,25}

\newcommand{\norm}[1]{\left \| #1 \right\|}
 
\newcommand{\Edist}[2]{\mathbb{E}_{#1}\left[#2\right] } 

\newcommand{\fopt}{f^*}

\newcommand{\xopt}{x^*}

\newcommand{\algstyle}{\sf}

\newcommand{\gd}{{\algstyle GD}}
\newcommand{\sgd}{{\algstyle SGD}}

\newcommand{\adagrad}{{\algstyle AdaGrad}}

\newcommand{\spsmax}{{\algstyle SPSmax}}

\newcommand{\decsps}{{\algstyle DecSPS}}

\newcommand{\libsvm}{{\algstyle LIBSVM}}
\newcommand{\psps}{{\algstyle PSPS}}
\newcommand{\sania}{{\algstyle SANIA}}
\newcommand{\sls}{{\algstyle SLS}}
\newcommand{\kate}{{\algstyle KATE}}
\newcommand{\momo}{{\algstyle MoMo}}

\newcommand{\pof}[1]{Proof of #1.}

\newcommand{\g}{\nabla f}

\newcommand{\mO}{\mathbf 0}

\newcommand{\normsM}[2]{{\left \| #1 \right\|}^2_{#2}}

\usepackage{theoremref}

\usepackage[colorinlistoftodos,bordercolor=orange,backgroundcolor=orange!20,linecolor=orange,textsize=scriptsize]{todonotes}

\newcommand{\sumin}[2]{ \sum \limits_{#1=1}^{#2}}
\newcommand{\avein}[2]{\frac 1 {#2} \sum \limits_{#1=1}^{#2}}
\newcommand{\avesn}[2]{\frac 1 {\vert #2 \vert} \sum \limits_{#1 \in #2}}
\newcommand{\prodin}[2]{ \prod \limits_{#1=1}^{#2}}
\newcommand{\norms}[1]{\norm{#1}^2}
\newcommand{\xerr}[1]{x^{#1} -\xopt}

\newcommand{\N}{\mathbb N}

\usepackage{enumitem}

\newcommand{\lr}{\left(} 
\newcommand{\rr}{\right)}
\newcommand{\ls}{\left[} 
\newcommand{\rs}{\right]}
\newcommand{\lc}{\left\lbrace} 
\newcommand{\rc}{\right\rbrace}
\newcommand{\la}{\left\langle} 
\newcommand{\ra}{\right\rangle}
\newcommand{\lv}{\left \vert}
\newcommand{\rv}{\right \vert}

\newcommand{\indep}{\perp \!\!\! \perp}

\newcommand{\footremember}[2]{%
    \footnote{#2}
    \newcounter{#1}
    \setcounter{#1}{\value{footnote}}%
}
\newcommand{\footrecall}[1]{%
    \footnotemark[\value{#1}]%
} 

\usepackage[verbose=true,letterpaper]{geometry}
\AtBeginDocument{
  \newgeometry{
    textheight=9in,
    textwidth=5.5in,
    top=1in,
    headheight=12pt,
    headsep=25pt,
    footskip=30pt
  }
}

    \author{\authors}
    \date{}
    \title{\mytitle}

\newcommand{\mytitle}{Polyak Stepsize: Estimating Optimal Functional Values Without Parameters or Prior Knowledge}

\newcommand{\mom}{\alpha}
\newcommand{\spc}{\gamma}
\newcommand{\ssize}{\eta}
\newcommand{\tp}{{\algstyle TP}}
\newcommand{\stp}{{\algstyle STP}}
\newcommand{\stpm}{{\algstyle STPm}}

\newcommand{\dog}{{\algstyle DoG}}
\newcommand{\tdog}{{\algstyle T-DoG}}
\newcommand{\dowg}{{\algstyle DoWG}}
\newcommand{\alig}{{\algstyle ALI-G}}

\newcommand{\xydiff}{{\color{blue} a}}
\newcommand{\yodiff}{{ \color{blue} b}}

\newcommand{\batch}{B}
\newcommand{\fbi}[1]{f_{\batch_{#1}}}
\newcommand{\gbi}[1]{\g_{\batch_{#1}}}

\title{\mytitle}

\author{
\and Farshed Abdukhakimov \\ \texttt{MBZUAI}\footremember{email}{Mohamed bin Zayed University of Artificial Intelligence. Email: \{first name\}.\{last name\}@mbzuai.ac.ae}
\and Cuong Anh Pham \\ \texttt{MBZUAI}\footrecall{email}
\and Samuel Horv\'ath \\ \texttt{MBZUAI}\footrecall{email}
\and Martin Tak\'a\v{c} \\ \texttt{MBZUAI}\footrecall{email}
\and Slavom\'ir Hanzely \\ \texttt{MBZUAI}\footrecall{email}
}

\begin{document}

\maketitle

\begin{abstract}
  The Polyak stepsize for Gradient Descent is known for its fast convergence but requires prior knowledge of the optimal functional value, which is often unavailable in practice.
    In this paper, we propose a parameter-free approach that estimates this unknown value during the algorithm’s execution, enabling a parameter-free stepsize schedule. Our method maintains two sequences of iterates: one with a higher functional value is updated using the Polyak stepsize, and the other one with a lower functional value is used as an estimate of the optimal functional value.
    We provide a theoretical analysis of the approach and validate its performance through numerical experiments. The results demonstrate that our method achieves competitive performance without relying on prior function-dependent information.
\end{abstract}

\section{Introduction }

In this paper, we consider the unconstrained minimization problem of a function $f: \R^d \to \R$,
\begin{equation} \label{eq:loss}
    \min_{x \in \R^d} f(x),
\end{equation}
where dimensionality $d$ is potentially large. We assume that $f$ is ill-conditioned, making hyperparameter tuning particularly challenging. This is a common issue in machine learning when datasets are not properly preprocessed.
A standard approach to solving \eqref{eq:loss} is Gradient Descent (\gd{}), a widely used optimization algorithm due to its simplicity and effectiveness across various domains, including machine learning, statistical modeling, and engineering:
\begin{align} \label{eq:gd}
    x^{k+1} = x^k - \ssize \g(x^k). \tag{GD}
\end{align}

Despite its popularity, traditional Gradient descent (\gd{}) faces significant challenges in selecting an appropriate stepsize, as the fine-tuning process can be both time-consuming and computationally expensive. The stepsize directly affects the convergence speed and stability of the optimization algorithm: a value that is too large leads to divergence, while one that is too small results in slow convergence, limiting the algorithm’s practical utility (see \Cref{fig:sensitivity_sgd}).
For example, consider a simple quadratic function $\frac 2L \norms x$, where the optimal stepsize is given by $  \ssize^*_{\gd{}} \eqdef \frac L2$. Choosing a smaller stepsize ($\ssize<\ssize^*_{\gd{}}$) slows down \gd{} linearly, whereas a larger stepsize $\ssize>2 \ssize^*_{\gd{}}$ causes divergence.
Stepsize selection is an important aspect of the training process, even for large models, \citet{schaipp2023momo} highlights the sensitivity of the learning rate in neural networks.

\begin{figure*}
    \centering
    \begin{subfigure}{0.28\linewidth}
        \centering
        \includegraphics[width=\linewidth]{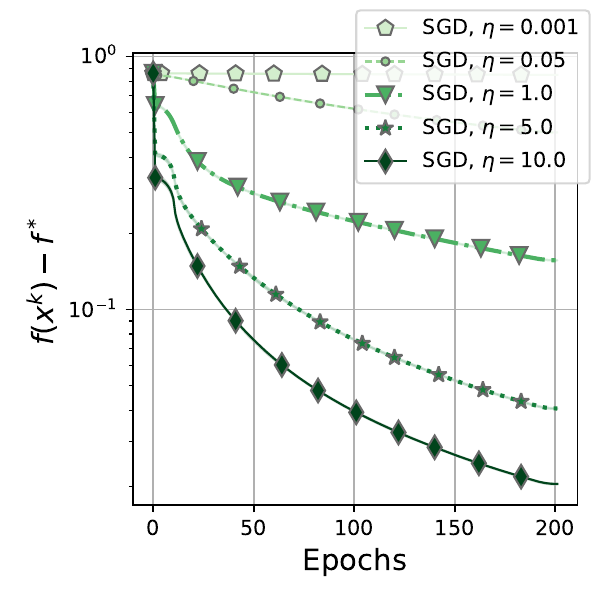}
        \caption{Logistic regression \eqref{eq:logistic} on \textbf{colon-cancer}.}
        \label{fig:sensitivity_sgd}
    \end{subfigure}
    \hfill
    \begin{subfigure}{0.28\linewidth}
        \centering
        \includegraphics[width=\linewidth]{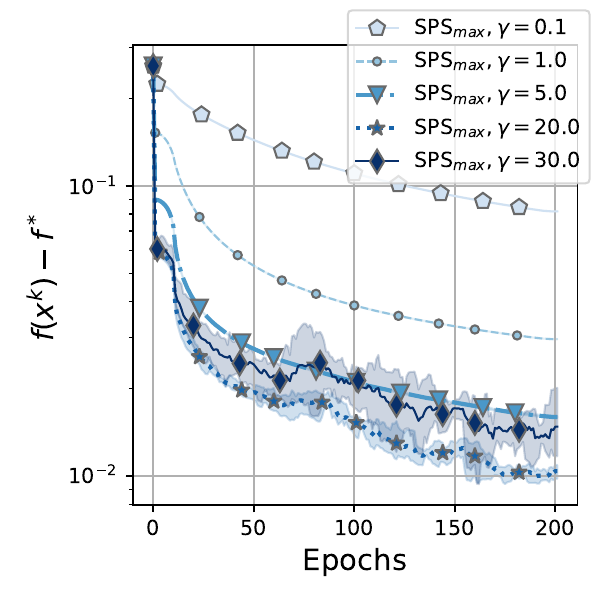}
        \caption{Logistic regression \eqref{eq:logistic} on \textbf{a1a}.}
        \label{fig:sensitivity_sps}
    \end{subfigure}
    \hfill
    \begin{subfigure}{0.28\linewidth}
        \centering
        \includegraphics[width=\linewidth]{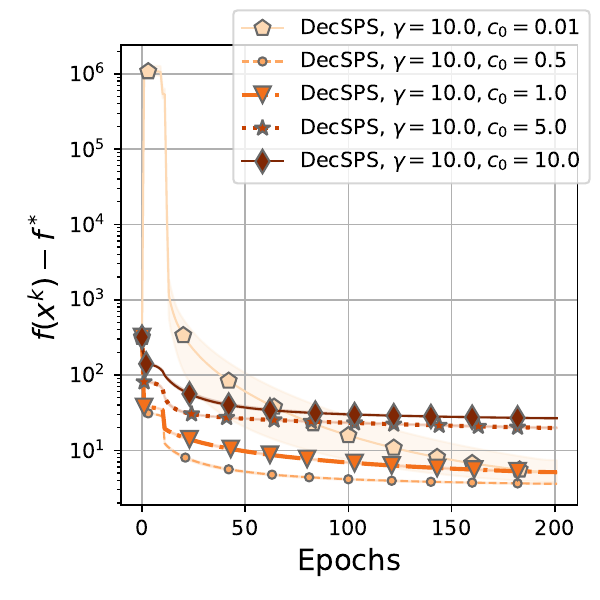}
        \caption{Least squares \eqref{eq:lstsq} on \textbf{housing}.}
        \label{fig:sensitivity_decsps}
    \end{subfigure}
    \caption{Sensitivity of stochastic methods \sgd{}, \spsmax{}, and \decsps{} to their respective hyperparameters on \libsvm{} datasets \citep{Chang2011LIBSVM}. We repeat the experiments for 5 different random seeds and plot mean $\pm$ standard deviation. For easier readability, we reduce variance by plotting running average of functional subpotimality.}
    \label{fig:sensitivity}
    \vspace{-10pt}
\end{figure*}



Due to its practical significance, the development of methods that automatically adapt stepsizes to a problem's geometry remains an active research topic up to this day. From the most recent development,
\citet{carmon22making} proposed a bisection procedure to make Stochastic Gradient Descent (\sgd{}) parameter-free via a backtracking procedure.
\citet{ivgi2023dog} introduced a parameter-free \sgd{} stepsize schedule called \dog{}, which sets the stepsize using empirically observed quantities -- as the ratio between the maximum observed distance from the initial point and an accumulation of gradient norms similar to \adagrad{} \citep{duchi2011adaptive, ward2020adagrad}.
Notably, \dog{} can diverge \citep{ivgi2023dog}, and to address this, the authors propose a modified variant, \tdog{}, which ensures convergence but is not entirely parameter-free.
\citet{khaled2023dowg} introduce a weighted variant, \dowg{}, which improves test accuracy in neural networks. While these methods are parameter-free and do not require prior knowledge of the function, they remain sensitive to the problem's scaling and, therefore, still depend on dataset preprocessing.
\citet{vaswani2019painless} propose using line-search methods, such as Armijo line-search, to automatically determine the stepsize for SGD in settings that satisfy the interpolation condition. Their method, \sls{}, achieves faster convergence and improved generalization when training deep neural networks for classification tasks.
AI-SARAH, proposed by \citet{shi2023ai}, extends the stochastic recursive gradient algorithm SARAH \cite{nguyen2017sarah,nguyen2021inexact,beznosikov2021random} by choosing a step size that minimizes the stochastic gradient magnitude of the variance reduced method.


\subsection{Polyak stepsize}
Apart from the most recent development, there has been a rich history of theoretically-based stepsize adaptations aiming to remove the need for hyperparameter tuning.
As one of the most prominent works, \citet{polyak1969minimization} analyzed quadratic functions and proposed an optimal stepsize for \gd{} that does not depend on the smoothness constant,
\begin{equation} \label{eq:polyak}
\ssize_k = \frac {f(x^k)-\fopt}{\norms {\g(x^k)}}. \tag{Polyak}
\end{equation}
The Polyak stepsize, which adapts based on the objective function value, offers a promising approach to automating stepsize selection. However, the Polyak stepsize requires prior knowledge of the optimal function value, $\fopt$, limiting its applicability to cases where this value is known. Follow-up works proposed techniques to remove this limitation while preserving asymptotic convergence to the solution or its neighborhood
\citep{goffin1998convergence, sherali2000variable, nedic2001incremental}.
Recent work \citet{hazan2019revisiting} proposed a method for estimating $\fopt$ with non-asymptotic convergence. Their approach starts with a lower bound on $\fopt$ and runs the optimization for multiple epochs. If the method fails to converge to a required precision, it restarts with an improved lower bound. While theoretically sound, this approach requires a complex criterion to determine whether the algorithm has converged or if another restart is necessary, which, in turn, demands prior knowledge of the function. 
For the empirical risk minimization problem, \citet{loizou2021stochastic} proposed an enhancing Polyak stepsize with the clipping \footnote{For a function $f$ with an empirical risk structure $f(x)=\avein in f_i(x)$, Stochastic Gradient Descent (\sgd{}) updates iterates as $x^{k+1}=x^k-\ssize \g_i(x^k)$ for a randomly sampled function $f_i$.}. After sampling a function $f_i$, their method \spsmax{} uses stepsize
\begin{equation} \label{eq:spsmax}
\ssize_k = \min \lc \frac {f_i(x^k)-\fopt_i}{c\norms {\g_i(x^k)}}, \spc \rc, \tag{SPSmax}
\end{equation}
where $\fopt_i$ denotes the minimum of function $f_i$, $\fopt_i \eqdef \min_{x \in \R^d} f_i(x)$. Notably, \spsmax{} improves practical performance for preprocessed datasets. However, it still relies on knowledge of the optimal functional values $\fopt_i$ and requires the selection of an additional parameter $\spc$ -- in \Cref{fig:sensitivity_sps}, we illustrate that \spsmax{} is sensitive to the introduced parameter $\spc$.
\cite{orvieto2022dynamics} relaxed the requirement for exact optimal functional values $\fopt_i$ by replacing it with lower bounds $l_i^*\leq \fopt_i$ and selecting stepsize recursively as
\begin{equation} \label{eq:decsps}
\ssize_k= \frac 1 {c_k}\min \lc \frac {f_i(x^k)-l^*_i}{\norms {\g_i(x_k)}}, c_{k-1} \ssize_{k-1} \rc. \tag{DecSPS}
\end{equation}
However, \decsps{} remains sensitive to the choice of parameters $c_k$. \citet{orvieto2022dynamics} recommended setting these hyperparameters as either $c_k=c_0 (k+1)$ or $c_k=c_0 \sqrt{k+1}.$ Nonetheless, as shown in \Cref{fig:sensitivity_decsps}, the initial choice of $c_0$ significantly impacts the convergence behavior. 

Furthermore, \citet{barre2020complexity} enhanced the Polyak stepsize by incorporating momentum, while \citet{takezawa2024parameter} established a connection between Polyak stepsizes and gradient clipping. In the interpolation regime, \citet{berrada2020training} proposed increasing the denominator of \spsmax{} by a constant, leading to the \alig{} method, and \citet{gower2022cuttings} interpreted the family of \spsmax{}-like methods as passive-aggressive methods. \citet{gower2025analysis} analyze an idealized method {\algstyle SPS*} that utilizes stochastic gradients in the solution and provides convergence rates and show its optimality for the Lipschitz functions.

\subsection{Requirement of parameter tuning}

Despite all of the interest of the optimization community, all of the aforementioned methods require hyperparameters tuning, and no single choice of these hyperparameters performs well across all functions. This is a fundamental limitation -- as the optimal hyperparameters depend on the properties of the function, selecting them indirectly involves incorporating the prior knowledge into the optimization algorithm.

For instance, the parameter $\spc$ in \spsmax{} is sensitive to dataset scaling -- even a simple loss transformation $f \to c \cdot f$ directly affects its optimal value. 
Similarly, the lower bound estimates $l^*_i$ are sensitive to dataset translation, even under a simple loss transformation $f \to c+ f$.\\

We propose a novel algorithm that leverages the Polyak stepsize schedule for automated and efficient stepsize selection. Our method is entirely parameter-free, scaling-invariant, and translation-invariant, and adapts dynamically to optimization landscapes without requiring any hyperparameter tuning.
By eliminating the need for preset parameters, our approach simplifies implementation and reduces the computational overhead associated with hyperparameter selection.

\subsection{Contributions}
Our contributions can be summarized as follows:
\begin{itemize}[leftmargin=*]
    \item \textbf{Drawbacks of parameter-dependent methods:}
    We analyze the limitations of existing approaches that require hyperparameter tuning or prior assumptions about the loss function.
    
    \item \textbf{Novel parameter-free algorithm:} 
    We introduce a novel algorithm \tp{}, a truly parameter-free variant of the Polyak stepsize (\Cref{alg:tp}). Unlike previous methods, our approach does not require any prior knowledge of the function.

    Our algorithm estimates the optimal functional value using an auxiliary sequence of iterates. To the best of our knowledge, this is the first work using this technique.

    \item \textbf{Convergence guarantees:} 
    We provide convergence guarantees under different sets of assumptions:
    \begin{itemize}
    \item Smoothness (\Cref{def:smoothness}) and strong convexity (\Cref{def:convexity}) in \Cref{th:strongly_convex},
    \item Bounded gradient assumption (\Cref{def:gbound}) and convexity (\Cref{def:convexity}) in \Cref{le:convex_gbound}.
    \end{itemize}

    \item \textbf{Stochastic loss:} 
    To align with the requirements of machine learning, we extend our method \tp{} to settings where the loss function is stochastic. 
        Our adaptation naturally replaces gradients with stochastic gradients (\Cref{alg:stp}).

    \item \textbf{Experimental comparison:}
    We evaluate the performance of our algorithm on various datasets and demonstrate that, despite being parameter-free, it achieves convergence comparable to methods that rely on parameter tuning. All our results are fully reproducible using the provided source code.
\end{itemize}

Motivated by similar objectives, we first review approaches that aim to automatically adapt to local geometry. Then, we introduce our algorithm along with its theoretical foundation. Next, we present experimental results demonstrating its practical performance. Finally, we conclude with a discussion of key findings and potential directions for future work.

\subsection{Adaptive methods}

Aiming to adapt to a local geometry \citet{duchi2011adaptive} proposed a family of subgradient methods including \adagrad{}. 
\citet{mishchenko2023prodigy} proposed a method of estimating the distance of the current iterate to the solution, which is used as a hyperparameter in various algorithms.
\citet{abdukhakimov2024stochastic} presented \psps{}, a parameter-free extension of {\algstyle{SPS}} that employs preconditioning techniques to handle poorly scaled and/or ill-conditioned datasets.
\citet{abdukhakimov2023sania} introduced \sania{}, a general optimization framework that unifies and extends well-known first-order, second-order, and quasi-Newton methods from the Polyak stepsize perspective, making them adaptive and parameter-free within the {\algstyle{SPS}} setting.
\citet{choudhury2024remove} proposed \kate{}, an adaptive and scale-invariant variant of \adagrad{}. \kate{} eliminates the square root in the denominator of the stepsize and introduces a new sequence in the numerator to compensate for its removal.
\citet{li2022sp2} proposed {\algstyle{SP2}}, a second-order Polyak step-size method that leverages second-order information to accelerate convergence.

\section{Algorithm}
\label{subsec:novel-algorithm}
After reviewing the relevant literature, we now introduce a novel algorithm inspired by the Polyak stepsize. Our method estimates the optimal functional value $\fopt$ using a secondary sequence of iterates, enabling automated stepsize selection without requiring any prior knowledge.
\begin{align}
    (x^{k+1}, y^{k+1}) \eqdef \nonumber
    \begin{cases}
    \lr x^k- 2\frac {f(x^k)-f(y^k)}{\norms {\g(x^k)}} \g(x^k), y^k \rr, & \text{if }f(x^{k})> f(y^{k})\\
    \lr x^k, y^k- 2\frac {f(y^k)-f(x^k)}{\norms {\g(y^k)}} \g(y^k) \rr, & \text{if }f(y^{k})> f(x^{k})
    \end{cases}. \tag{\tp{}}
\end{align}
We call this algorithm \emph{Polyak method with twin iterate sequences} (or Twin Polyak, \tp{}), as detailed in \Cref{alg:tp}.
By maintaining two sequences of iterates, \tp{} updates the sequence with the higher functional value using an estimate of the optimal functional value provided by the other sequence.

To the best of our knowledge, \tp{} is the first variant of the Polyak method to adopt this technique. 
In contrast to existing Polyak method variants that estimate the optimal functional value $\fopt$ using lower bounds \citep{hazan2019revisiting, loizou2021stochastic, orvieto2022dynamics}, \tp{} uses the second sequence to estimate the optimal functional value from above. \\
\begin{algorithm} 
    \caption{\tp{}: Polyak method with twin iterate sequences} \label{alg:tp}
    \begin{algorithmic}[1]
        \State \textbf{Requires:} Initial points $x^0, y^0 \in \R^d,$ constant $\varepsilon>0$.
        \For {$k=0,1,2\dots$}
            \State Compute $f(x^k)$ and $f(y^k)$
            \If {$\vert f(x^{k}) - f(y^{k}) \vert < \varepsilon$}
                \State Return $x^k$
            \ElsIf {$f(x^{k})>f(y^{k})$}
                \State $x^{k+1} = x^k- 2\frac {f(x^k)-f(y^k)}{\norms {\g(x^k)}} \g(x^k)$
                \Comment{Step for $x$}
                \State $y^{k+1} = y^k$
            \Else
                \State $x^{k+1} = x^k$
                \State $y^{k+1} =  y^k- 2\frac {f(y^k)-f(x^k)}{\norms {\g(y^k)}} \g(y^k)$
                \Comment{Step for $y$}
            \EndIf
        \EndFor
    \end{algorithmic}
\end{algorithm}

\subsection*{Properties of \tp{}}
Observe that the \tp{} algorithm is less dependent on dataset preprocessing, as it remains invariant to both scaling and translation of the dataset (i.e., the sequences of iterates are unchanged under transformations $f \to c \cdot f$ and $c+f$).

For quadratic functions, \tp{} alternates updates between the sequences $x^k$ and $y^k$ and converges linearly to the solution. This property is formalized in the following lemma.

\begin{lemma} \label{le:quadratic}
    Consider quadratic function $f(x) \eqdef \frac 12 \normsM x 2 + b$ and updates of \tp{} (\Cref{alg:tp}) with initialization points $x^0, y^0$ such that $f(x^0)\not = f(y^0)$.
    Without loss of generality, consider $f(x^{k})> f(y^{k})$. Then it holds
    \begin{align}
    \frac {\normsM {x^{k+1}} 2 } {\normsM {y^{k+1}} 2} 
    = \frac {\normsM {y^k} 2 } {\normsM {x^k} 2}.
    \end{align}
    Therefore, models change their relative position each step, $f(x^{k+1})< f(y^{k}),$ and \tp{} converges linearly to the solution $x^*=\mO$ with the factor $ {\normsM {y^0} 2 } /  {\normsM {x^0} 2}$.
\end{lemma}

\section{Theory}
\label{sec:theory}
Before presenting our theoretical guarantees, we first review standard results for Polyak stepsizes.\\ 


We will establish convergence results under different sets of standard assumptions, so we need to define the notions of convexity, strong convexity, Lipschitz smoothness, and $G$-bounded gradients.

Strong convexity provides a quadratic lower bound on the functional values, while convexity provides a weaker linear lower bound.
\begin{definition} \label{def:convexity}
    A differentiable function $f: \R^d \to \R$ is called \emph{$\mu$-strongly convex} if there exist a constant $\mu\geq 0$ so that for all $ x,y \in \R^d$ holds
    \begin{equation} \label{eq:strongly_convex}
        f(y)\geq f(x)+\la \g(x), y-x \ra +\frac \mu 2 \norms{y-x}.
    \end{equation}
    If \eqref{eq:strongly_convex} holds with $\mu=0$, we say that function $f$ is \emph{convex}.
\end{definition}

Lipschitz smoothness provides a quadratic upper bound on the functional values.
\begin{definition} \label{def:smoothness}
    A differentiable function $f: \R^d \to \R$ is called \emph{$L$-smooth} if there exist a constant $L\geq 0$ so that for all $ x,y \in \R^d$ holds
    \begin{equation}
        f(y)\leq f(x)+\la \g(x), y-x \ra +\frac L 2 \norms{y-x}.
    \end{equation}
\end{definition}

The assumption of $G$-bounded gradients constrains the gradient magnitudes, effectively limiting the size of the updates.
\begin{definition} \label{def:gbound}
    A differentiable function $f: \R^d \to \R$ is called \emph{$G$-bounded} in the set $S$, if the supremum of the gradient norms within $S$ is finite, 
    \begin{equation}
        G \eqdef \sup_{x \in S} \norm{\g(x)} <\infty.
    \end{equation}
\end{definition}
However, a uniform gradient bound does not always exist. For instance, quadratic functions can exhibit arbitrarily large gradients when initialized sufficiently far from the solution. In fact, this assumption is mutually exclusive with strong convexity. Consequently, determining the constant $G$ for $G$-bounded gradients is often nontrivial.

The convergence of the Polyak method is well established. In the bounded gradient case it achieves the rate $\mathcal O \lr k^{-\frac 12} \rr$ and the case of strongly convex functions it achieves linear convergence.

\begin{proposition} \label{pr:polyak_rate}
    Gradient descent with Polyak stepsizes \eqref{eq:polyak}, under the convexity assumption and that the gradients are $G$-bounded converges with the rate that depends on $k$ as $\cOb {k^{-\frac 12}}$.
\end{proposition}
\begin{proposition}\label{pr:strongly_convex}
    If the function $f:\R^d \to \R$ is $\mu$-strongly convex (\Cref{def:convexity}) is minimized with Gradient descent with stepsizes chosen such that $0\leq \ssize_t \leq 2\frac {f(x^t)-\fopt}{\norms{\g(x^t)}}$ for all $t\leq k$.
    Then, after $k$ updates, the sequence $\{x^t\}_{t=0}^{k}$ has guaranteed decrease
    \begin{align}
        \norms{x^{k}-\xopt} 
        \leq \lr \prodin tk \lr 1-\mu \ssize_t \rr\rr \norms{x^0-\xopt}.
    \end{align}
\end{proposition}
These results are well-known; for reader's convenience we include the proofs 
in \Cref{sec:proof_pr_strongly_convex}.

\subsection{Guarantees for our algorithm}
Noting that our algorithm \tp{} uses stepsize smaller than Polyak stepsize,
\begin{equation} \label{eq:tp_polyak_bound}
    \ssize_{\tp{}} 
    \eqdef 2\frac {f(x^k)-f(y^k)}{\norms {\g(x^k)}} \g(x^k)
    \leq 2\frac {f(x^k)-\fopt}{\norms {\g(x^k)}} \g(x^k),
\end{equation}
\Cref{pr:strongly_convex} guarantees that \tp{} ensures a monotonic decrease in the functional values. Note that equality in  \eqref{eq:tp_polyak_bound} holds only when $y^k$ is already the solution of the optimization problem. 

Since our method sets stepsizes based on the difference between the model losses, our theoretical guarantees rely on the assumption that those losses remain sufficiently different, formalized below. 

\begin{assumption}\label{as:xydiff}
    There exists constants $1>\xydiff>0$ so that for sequences $x^k, y^k \in \R^d$ hold
    \begin{align}
         \lv f(x^k)-f(y^k) \rv
        &\geq \xydiff \ls \max \lc f(x^k), f(y^k) \rc -\fopt  \rs.
    \end{align}
\end{assumption}
\begin{assumption}\label{as:yodiff}
    There exists constants $1>\yodiff>0$ so that for iterates $x^k, y^k \in \R^d$ hold
    \begin{align}
        \min \lc f(x^k),f(y^k)\rc -\fopt
        &\geq \yodiff \ls \max \lc f(x^k), f(y^k) \rc -\fopt  \rs.
    \end{align}
\end{assumption}

For quadratic functions, both $\xydiff$ and $\yodiff$ are guaranteed to exist, as \Cref{le:quadratic} proves that ratio of functional values of sequences remains constant.
For the strongly convex loss, our algorithm converges linearly, as formalized in the theorem below. 

\begin{theorem} \label{th:strongly_convex}
     If the function $f:\R^d \to \R$ is $\mu$-strongly convex (\Cref{def:convexity}) and $L$-smooth (\Cref{def:smoothness}) and sequences $\{x^t\}_{t=0}^k, \{y^t\}_{t=0}^k$ generated by Twin-Polyak satisfy Assumption \ref{as:xydiff}, then stepsize of Twin-Polyak
    \begin{align}
        \ssize_k\eqdef 2\frac {f(x^k)-f(y^k)}{\norms{\g(x^k)}}
    \end{align}
    satisfies
    \begin{align}
        \ssize_k
        \geq 2\xydiff \frac {f(x^k)-\fopt}{\norms{\g(x^k)}} 
        \geq \frac \xydiff L 
    \end{align}
    and 
    we have guaranteed linear convergence $\norms{x^{k}-\xopt} \leq \lr 1- \xydiff \frac \mu L \rr^{k} \norms{x^0-\xopt}.$ For $\varepsilon>0$, Twin-Polyak reaches $\varepsilon$-neighborhood in $\cO\lr \frac 1 \xydiff \frac L \mu \log \frac 1 \varepsilon \rr$ iterations. 
\end{theorem}
\begin{proof}[\pof{\Cref{th:strongly_convex}}]
    Follows directly from \Cref{pr:strongly_convex}.
\end{proof}
For functions with $G$-bounded gradients, \tp{} converges with the rate depending on $k$ as $\cOb {k^{-\frac 12}}$.
\begin{lemma}\label{le:convex_gbound}
    If the function $f:\R^d \to \R$ is convex (\Cref{def:convexity}) and with $G$-bounded gradients (\Cref{def:gbound}) and sequences $\{x^t\}_{t=0}^k, \{y^t\}_{t=0}^k$ generated by Twin-Polyak with stepsize
    \begin{align}
        \ssize_k\eqdef 2\frac {f(x^k)-f(y^k)}{\norms{\g(x^k)}}.
    \end{align} 
    satisfies Assumptions~\ref{as:xydiff}, \ref{as:yodiff}. Then, after $k$ updates, we have the following guaranteed decrease:
    \begin{align}
        \min_{t\in\{1, \dots, k\}}  f(x^t)-\fopt
        &\leq \frac {G\norm{x^0-\xopt}}{2 \sqrt{\xydiff\yodiff k} }.
    \end{align}
\end{lemma}



\section{Stochastic gradients}

In modern machine learning, the loss function often takes the form of an empirical risk over data points, making the objective \eqref{eq:loss} an empirical risk minimization problem:
\begin{equation}
    \min_{x \in \R^d} \lc f(x) \eqdef \avein in f_i(x) \rc,
\end{equation}
where the number of data points $n$ is large. Instead of evaluating the full loss $f(x)$ and computing the full gradients $\g(x)$, it is often preferable to use a sample batch of stochastic gradients. 
Denoting $\batch$ a batch of indices corresponding to a subset of stochastic functions $\batch \in 2^{\{1, 2, \dots, n\}}$, we denote the batch functional value $\fbi{}(x) \eqdef \avesn i {\batch}f_i(x)$ and batch gradient $\gbi{}(x) \eqdef \avesn i {\batch}\g_i(x).$

\begin{remark}
In practice, the batch size is typically chosen to fully utilize GPU memory, making it a hardware-related parameter rather than one dictated by the loss.
Similarly, target accuracy $\varepsilon$ is often determined by the precision of arithmetic operations rather than the properties of the loss.
\end{remark}

We are going to propose a stochastic variant of \tp{} that replaces full gradients with stochastic gradients. This introduces stochasticity to both the gradient estimator and the stepsize schedule, making theoretical analysis much more challenging. For example, \Cref{le:sig} in Appendix demonstrates that utilizing stepsizes dependent on stochastic gradients makes it impossible to preserve gradient unbiasedness. Therefore, in the stochastic setup, we prioritize practical performance and implementability over theoretical guarantees. 

How can we obtain a stochastic variant of \tp{}? 
One possible approach is to directly replace the full loss/gradients with their stochastic counterpart (see \Cref{alg:stp}, \stp{}, in the appendix).
However, leveraging insights of \citet{schaipp2023momo}, we can enhance the naive algorithm with a momentum, leading to the algorithm \stpm{} (\Cref{alg:stpm}).
\begin{algorithm} 
    \caption{\stpm{}: Stochastic Twin-Polyak method with momentum} \label{alg:stpm}
    \begin{algorithmic}[1]
        \State \textbf{Requires:} Initial points $x^0, y^0 \in \R^d,$ target accuracy $\varepsilon>0$, batch size $\tau$, momentum parameter $\mom \in [0,1)$, number of steps $K$.
        \For {$k=0,1,2\dots,K-1$}
            \State Sample batch of stochastic functions $\batch_k$
            \State $\bar{f}_x^k = \mom \bar{f}_x^{k-1}+ (1-\mom) \fbi k(x^k) $
            \State $g_x^k = \mom g_x^{k-1}+ (1-\mom) \gbi k(x^k) $
            \State $z_x^k = \mom z_x^{k-1}+ (1-\mom) \la \gbi k(x^k), x^k \ra $
            \State $h_x^k = \bar{f}_x^k + \la g_x^k, x^k \ra - z_x^k $
            
            \State 
            
            \State $\bar{f}_y^k = \mom \bar{f}_y^{k-1}+ (1-\mom) \fbi k(y^k) $
            \State $g_y^k = \mom g_y^{k-1}+ (1-\mom) \gbi k(y^k) $
            \State $z_y^k = \mom z_y^{k-1}+ (1-\mom) \la \gbi k(y^k), y^k \ra $
            \State $h_y^k = \bar{f}_y^k + \la g_y^k, y^k \ra - z_y^k $

            \If {$h_x^k>h_y^k$}
                \State $y^{k+1} = y^k$
                \State $x^{k+1} = x^k- 2\frac {h_x^k-h_y^k}{\norms {g_x^k}} g_x^k$
            \Else
                \State $x^{k+1} = x^k$
                \State $y^{k+1} =  y^k- 2\frac {h_y^k-h_x^k}{\norms {g_y^k}} g_y^k$
            \EndIf
        \EndFor
    \State Return one of the points $(x^K, y^K)$ at random.
    \end{algorithmic}
\end{algorithm}

\section{Experiments} \label{sec:experiments}
We perform numerical comparative studies of a linear model solving binary classification and regression problems from the \libsvm{} dataset repository \citep{Chang2011LIBSVM}. We compare our method against baselines: \sgd{} with $\eta=1/L$, \spsmax{}, \decsps{}, and \sls{}. 

\begin{figure}
    \centering
    \begin{subfigure}{\textwidth}
        \centering
        \includegraphics[width=0.49\linewidth]{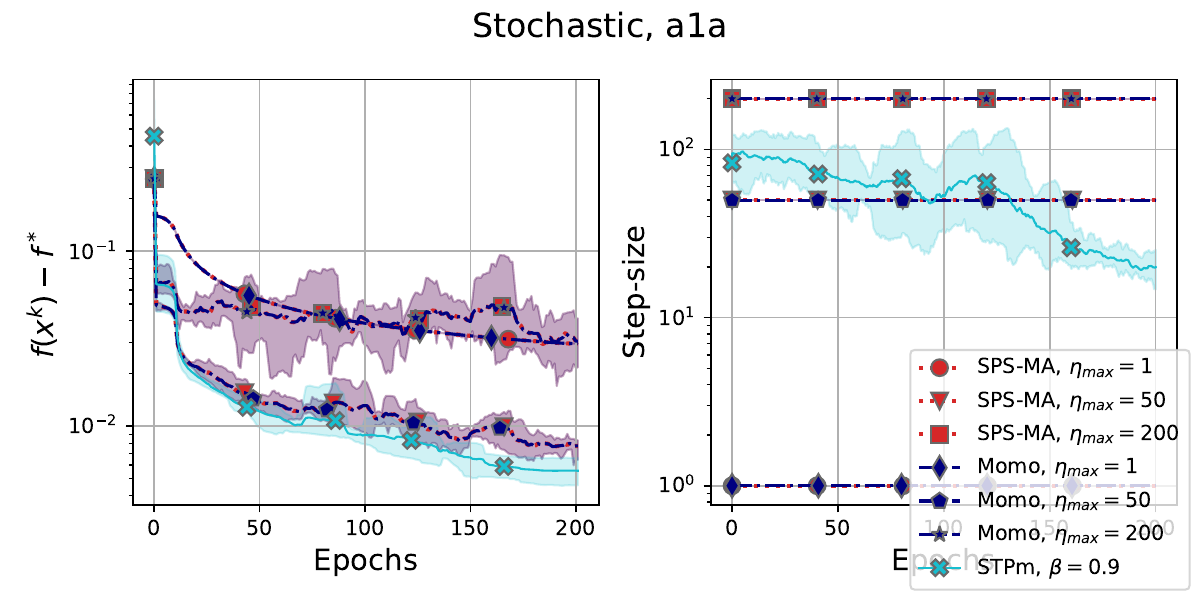}
        \hfill
        \includegraphics[width=0.49\linewidth]{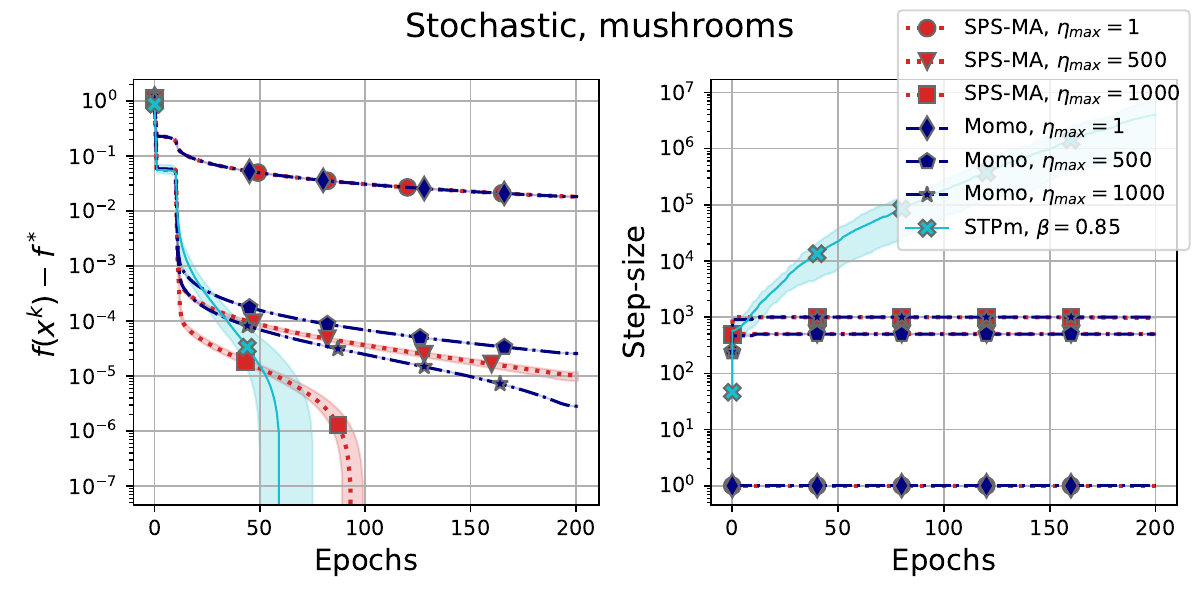}
        \includegraphics[width=0.49\linewidth]{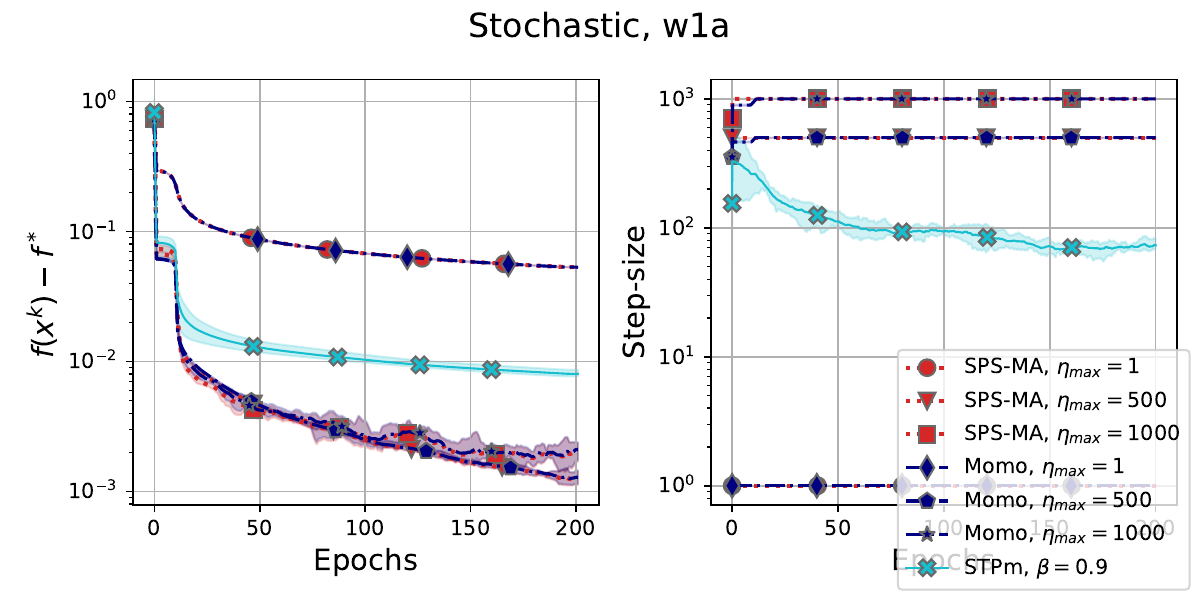}
        \hfill
        \includegraphics[width=0.49\linewidth]{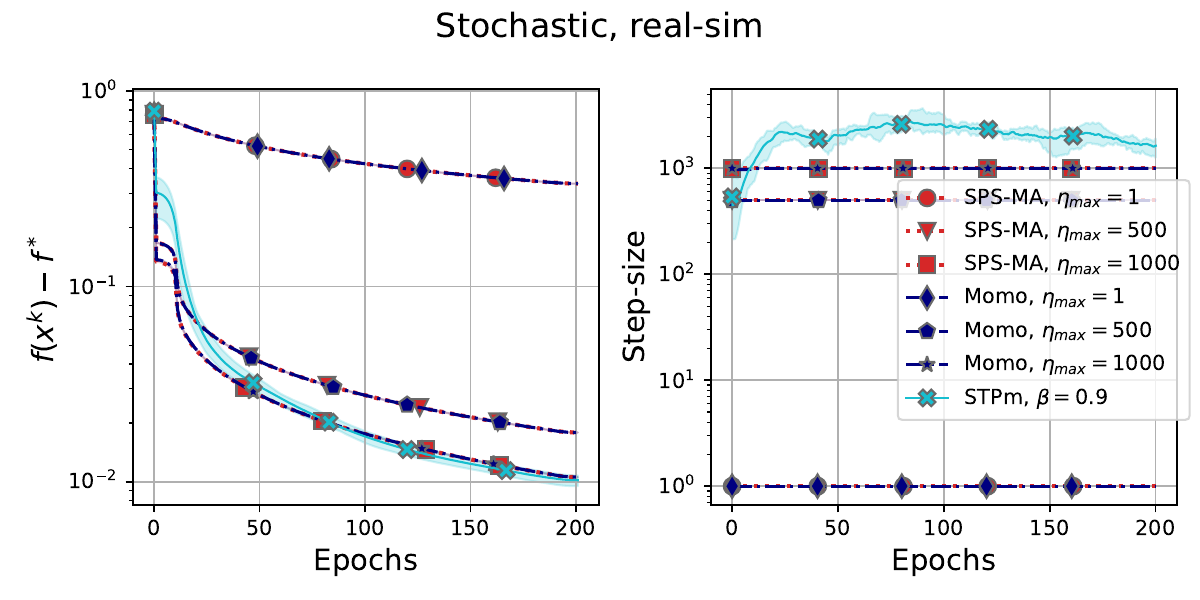}
    \end{subfigure}
    \caption{Performance and stepsize evolution of \stpm{} compared to other methods on binary classification problems minimizing \textbf{logistic regression} loss function.}
    \label{exp:stoch-logreg}
    \vspace{-10pt}
\end{figure}

For \spsmax{}, we fine-tune the stepsize bounding hyperparameter \(\gamma > 0\). Notably, \spsmax{} is highly sensitive to this bound and in many cases, \spsmax{} uses this bound from the beginning of training, effectively turning it into \sgd{} with \(\eta = \gamma\). Unsurprisingly, if $\gamma$ is chosen too large, it can cause divergence.
We select the recommended value \(c = 1/2\) for strongly convex functions.

Decreasing SPS (\decsps{}) has two hyperparameters: \( \eta_b \), a step-size bound similar to \spsmax{}, set by default to \( \eta_b = 10 \), and a scheduling scheme \( c_k = c_0 \sqrt{k+1} \), where \( c_0 = 1 \) and \( k \) is the current iteration. 
Notably, \decsps{} is still sensitive to the choice of \( c_0 \) (see \Cref{fig:sensitivity_decsps}).

As an additional baseline, we use Stochastic Line Search (\sls{}) that employs \textit{Armijo} line-search procedure with a step-size decrease factor of \( \beta = 0.9 \) and a sufficient decrease parameter of \( c = 0.1 \).  

We initialize all methods from the same initial point $x_0 \sim \N(\mathbf 0, \mathbf 1)$, and \stpm{} additionally samples a point $y_0 \sim \N(\mathbf 0, \mathbf 1)$. All experiments are repeated using five independent seeds. Additional numerical comparisons and technical details are provided in Appendix~\ref{app:additional-experiments}, and the source code is publicly available online\footnote{\url{https://github.com/fxrshed/twin-polyak}}. 

In \Cref{exp:stoch-logreg}, methods \spsmax{} and \momo{} are quickly reaching their stepsize upper bound  $\gamma$, effectively performing \sgd{} with stepsize $\eta=\gamma$ (disregarding the Polyak stepsize formula). This illustrates sensitivity of \spsmax{} and \momo{} to $\gamma$. 
\Cref{fig:lr-sweep} shows that parameter-free performance of the \stpm{} is comparable to baselines with well-tuned hyperparameters.
The experiments demonstrate that our methods achieve similar or better performance and even without any fine-tuning.


\begin{figure}
    \centering
    \begin{subfigure}{0.3\textwidth}
        \centering
        \includegraphics[width=\linewidth]{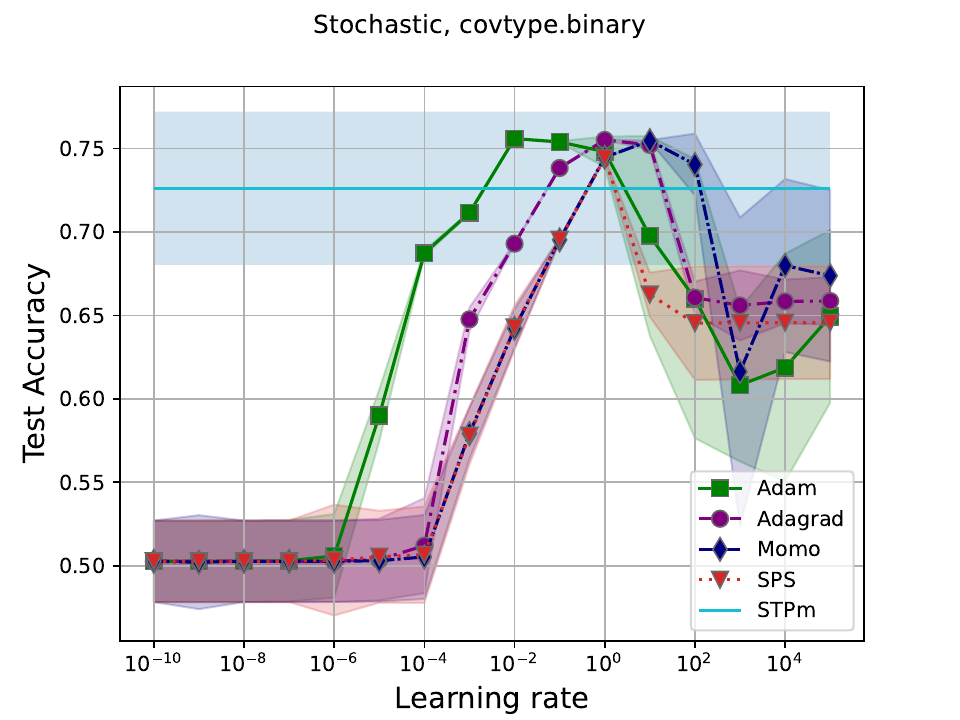}
        \caption{Linear model}
    \end{subfigure}
    \begin{subfigure}{0.3\textwidth}
        \centering
        \includegraphics[width=\linewidth]{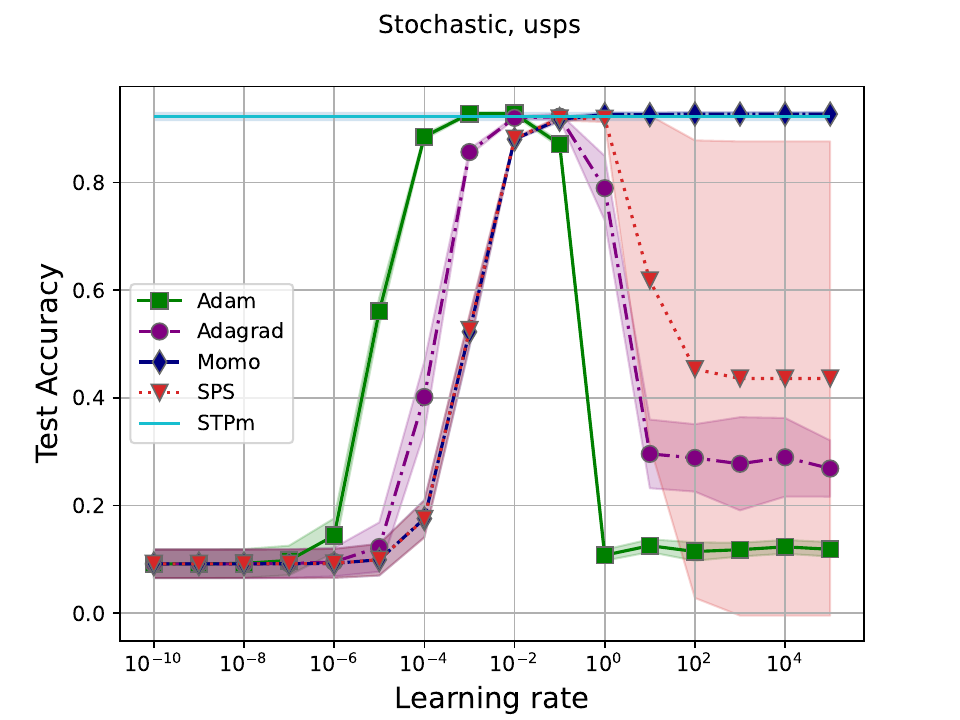}
        \caption{MLP}
    \end{subfigure}
    \begin{subfigure}{0.3\textwidth}
        \centering
        \includegraphics[width=\linewidth]{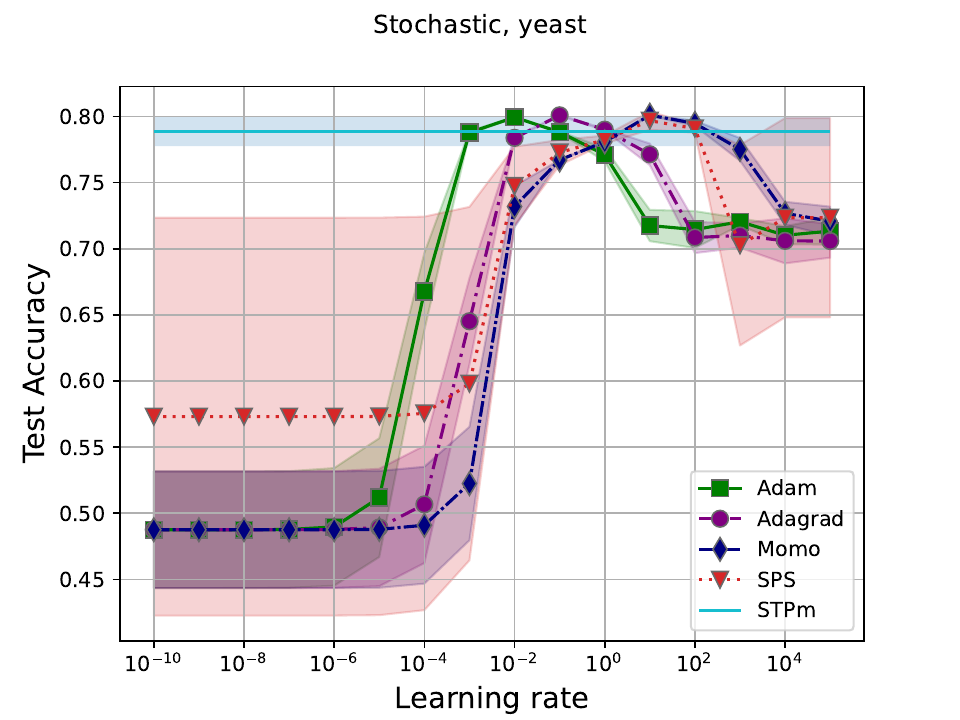}
        \caption{Linear model}
    \end{subfigure}
    \caption{Test accuracy after 50 epochs, for varying learning rate. \stpm{} does not use a learning rate. The shaded area depicts the mean $\pm$ standard deviation.}
    \label{fig:lr-sweep}
\end{figure}

\section{Limitations} \label{sec:limitations}
While \tp{} eliminates the need for hyperparameter tuning, it involves maintaining and updating two sequences of iterates, introducing additional overhead. Also, like many gradient-based optimization methods, the performance of \tp{} can be sensitive to the initialization of the iterates. Poor initialization of $x^0$ and $y^0$ might lead to suboptimal performance.

\section{Conclusion and future work}
In this paper, we introduced novel parameter-free optimization methods \tp{} and \stpm{} based on Polyak stepsizes. Our methods remove the impractical requirement of Polyak stepsizes to know of optimal functional value $f^*$ by estimating it via a complementary sequence of iterates.

In the setting where computing full gradients is feasible, \tp{} provably enjoys linear convergence for the strongly convex functions. 
For the practical machine learning setup where usage of stochastic gradients is preferred, we provide a numerical comparison of the method \stpm{} on binary classification and regression problems, and \stpm{} demonstrates competitive performance while being completely parameter-free. 

Further experimental validation on a wide range of machine learning tasks, such as training deep neural networks, reinforcement learning, or large-scale natural language processing tasks, is crucial. These empirical studies could help assess the robustness and generalization of the proposed methods across different problem domains and provide better insight into their real-world applicability.

\nocite{horvath2022adaptive}

\bibliographystyle{plainnat}
\bibliography{references}


\appendix

\clearpage

\section{Stochastic variant of \tp{}}
For the reader's convenience, we include a method that directly replaces full loss and gradients with their stochastic counterparts.
\begin{algorithm} 
    \caption{\stp{}: Stochastic Twin-Polyak method} \label{alg:stp}
    \begin{algorithmic}[1]
        \State \textbf{Requires:} Initial points $x^0, y^0 \in \R^d,$ target accuracy $\varepsilon>0$, batch size $\tau$, number of steps $K$.
        \For {$k=0,1,2\dots, K-1$}
            \State Sample batch of stochastic functions $\batch_k$
            \State Compute $\fbi k(x)$ and $\fbi k(y)$
            \If {$\vert \fbi k(x^{k}) - \fbi k (y^{k}) \vert < \varepsilon$}
                \State \textbf{continue}
            \ElsIf {$\fbi k(x^{k})>\fbi k(y^{k})$}
                \State $y^{k+1} = y^k$
                \State $x^{k+1} = x^k- 2\frac {\fbi k(x^k)-\fbi k(y^k)}{\norms {\gbi k(x^k)}} \gbi k(x^k)$
            \Else
                \State $x^{k+1} = x^k$
                \State $y^{k+1} =  y^k- 2\frac {\fbi k(y^k)-\fbi k(x^k)}{\norms {\gbi k(y^k)}} \gbi k(y^k)$
            \EndIf
        \EndFor
        \State Return $\argmin_{z\in x^K, y^K} f(z)$
    \end{algorithmic}
\end{algorithm}

\section{Additional experiments}
\label{app:additional-experiments}
In this section, we include the detailed description of experiments, additional experiments with supplementary performance metrics, such as $\|\nabla f(x^k) \|^2$ and accuracy on test data, for stochastic experiments. Also, we include deterministic experiments.  

\subsection*{Binary classification using logistic regression loss} 
We evaluate the performance of the proposed algorithms on a binary classification task by minimizing the logistic regression loss.
For data points ${\{(a_i,b_i)\}}_{i=1}^n,$ where $a_i \in \R^d, b_i \in \{-1, +1\}$ we aim to minimize
\begin{align}\label{eq:logistic}
    \min_{x \in \R^d }\Big\{ 
    f(x) &= 
    \frac{1}{n}\sum_{i=1}^n \log \lr 1 + e^{-b_i \langle a_i, x\rangle} \rr \Big\}.
\end{align}

\subsection*{Regression using least squares loss}
We evaluate the performance of the proposed algorithms on a regression task by minimizing the least squares loss.
For data points ${\{(a_i,b_i)\}}_{i=1}^n,$ where $a_i \in \R^d, b_i \in \R$ we aim to minimize
\begin{align}\label{eq:lstsq}
 \min_{x \in \R^d }\Big\{ 
    f(x) &= 
    \frac{1}{2n}\sum_{i=1}^n \lr\langle a_i, x \rangle - b_i \rr^2
    \Big\}.
\end{align}

\subsection*{Technical details}
All experiments were run with $5$ different seeds $(0, 1, 2, 3, 4)$ using \textit{NumPy} and \textit{PyTorch} on a computing machine with AMD EPYC 7402 24-Core Processor with 2.8GHz of the base clock. For deterministic experiments, an average of $5$ different seeds represents runs with 5 different $x_0$ and $y_0$ initializations for all methods.

\begin{figure*}
    \centering
    \begin{subfigure}{\textwidth}
        \centering
        \includegraphics[width=\linewidth]{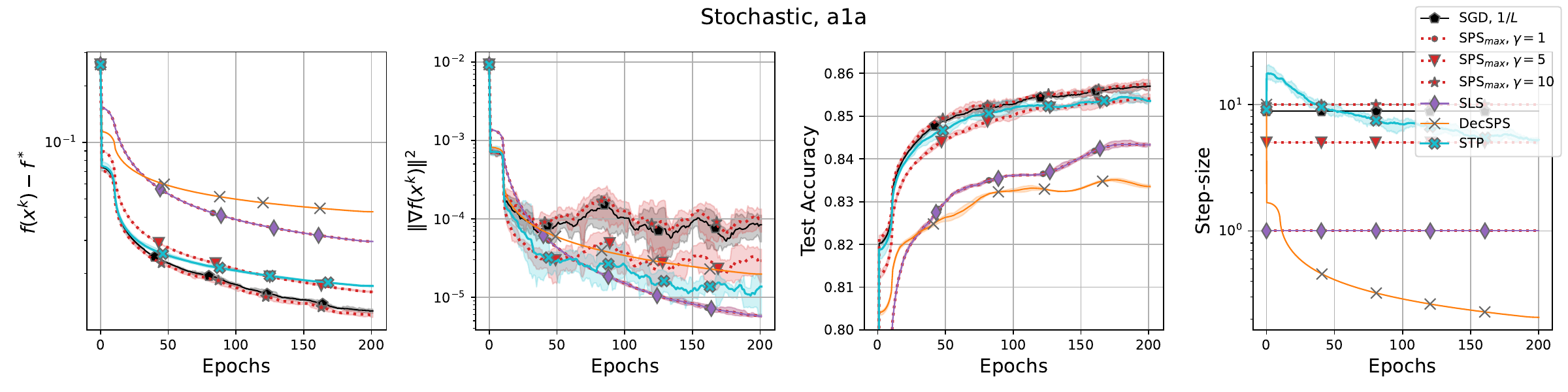}
        
        \includegraphics[width=\linewidth]{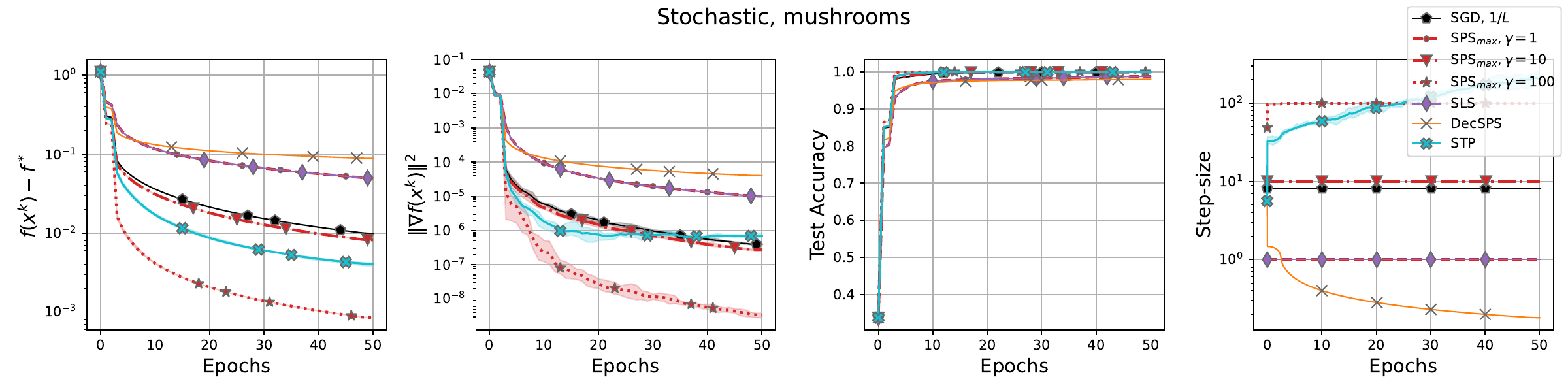}

        \includegraphics[width=\linewidth]{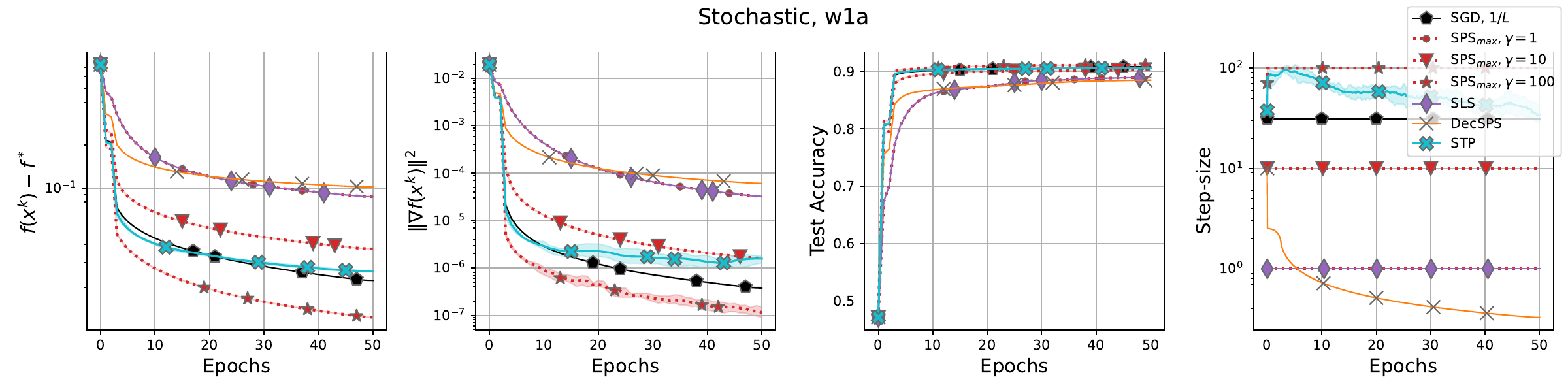}
        
        \includegraphics[width=\linewidth]{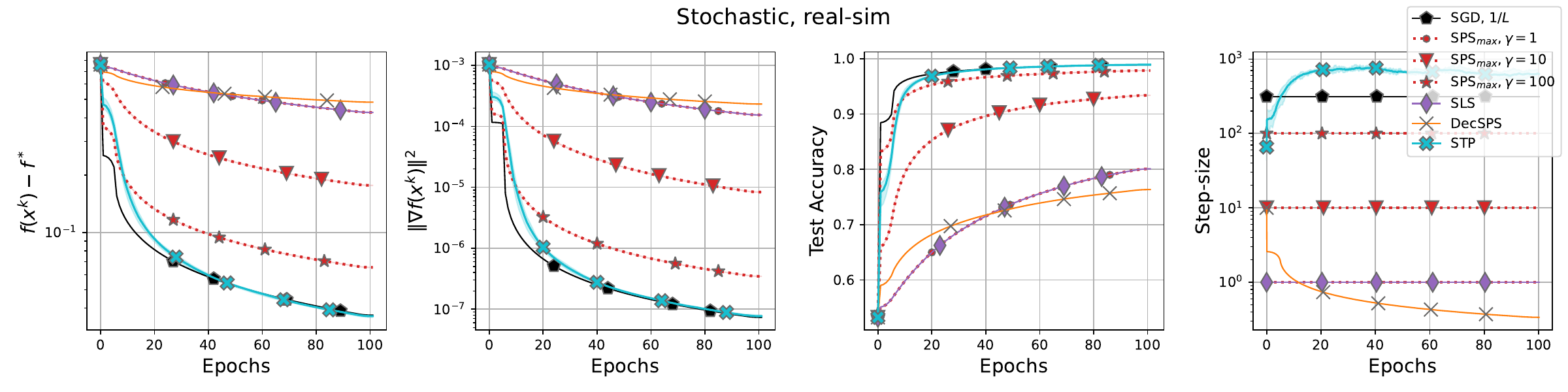}
    \end{subfigure}
    \caption{Performance and stepsize evolution of \stp{} compared to other methods on binary classification problems minimizing \textbf{logistic regression} loss function.}
    \label{exp:stoch-logreg-4}
\end{figure*}

\begin{figure*}
    \centering
    \begin{subfigure}{\textwidth}
        \centering
        \includegraphics[width=\linewidth]{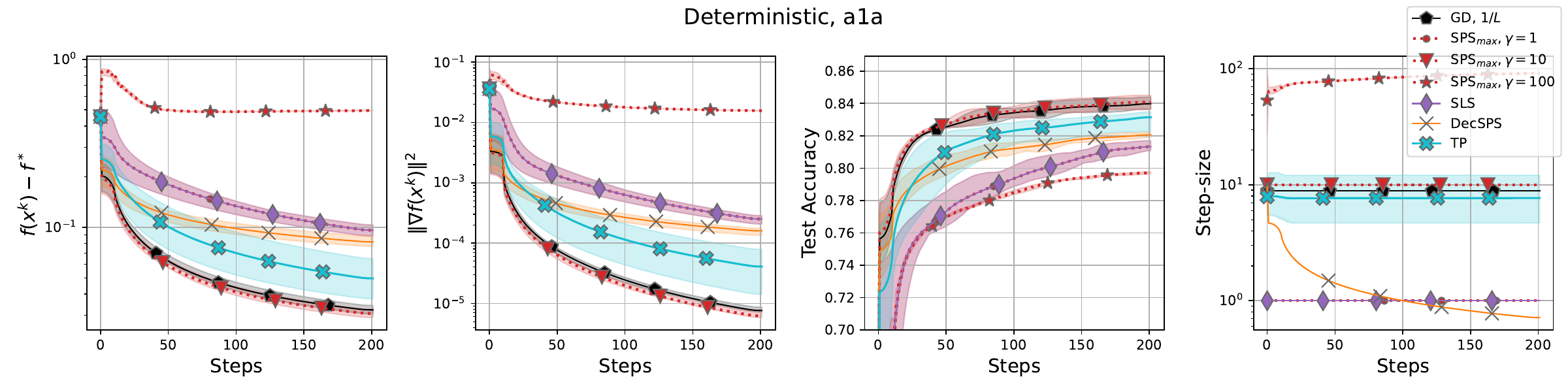}
        
        \includegraphics[width=\linewidth]{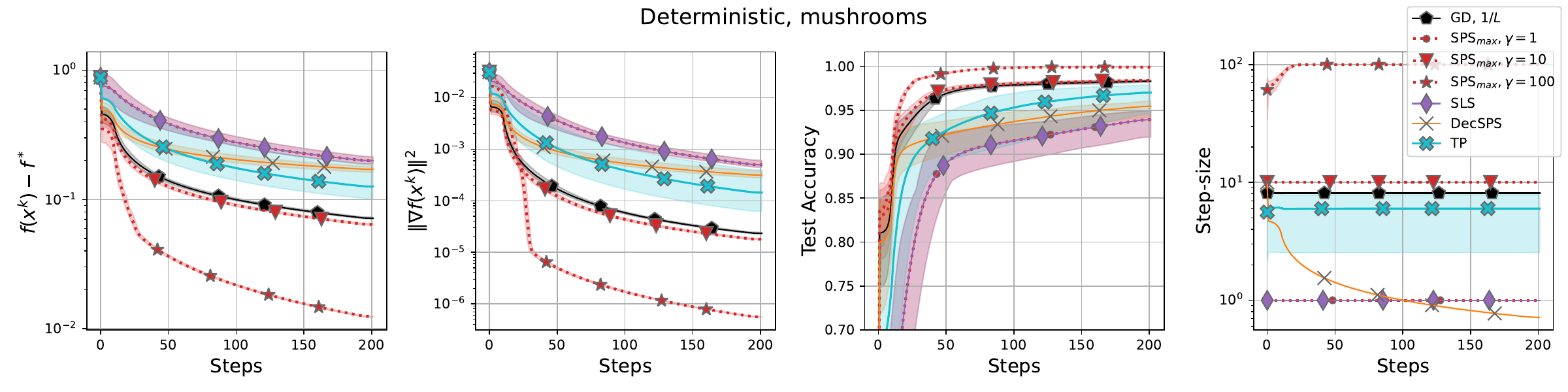}

        \includegraphics[width=\linewidth]{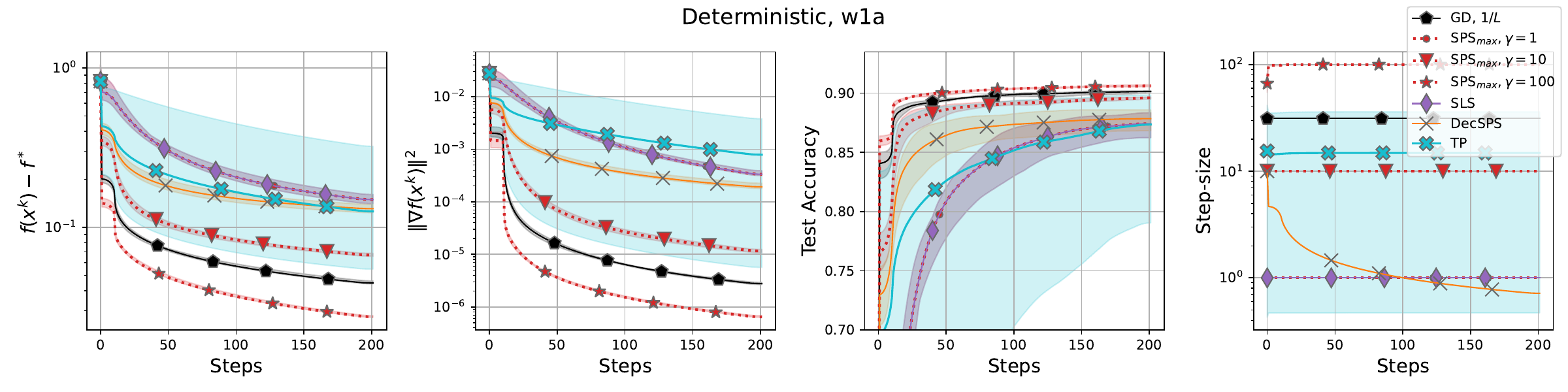}
        
        \includegraphics[width=\linewidth]{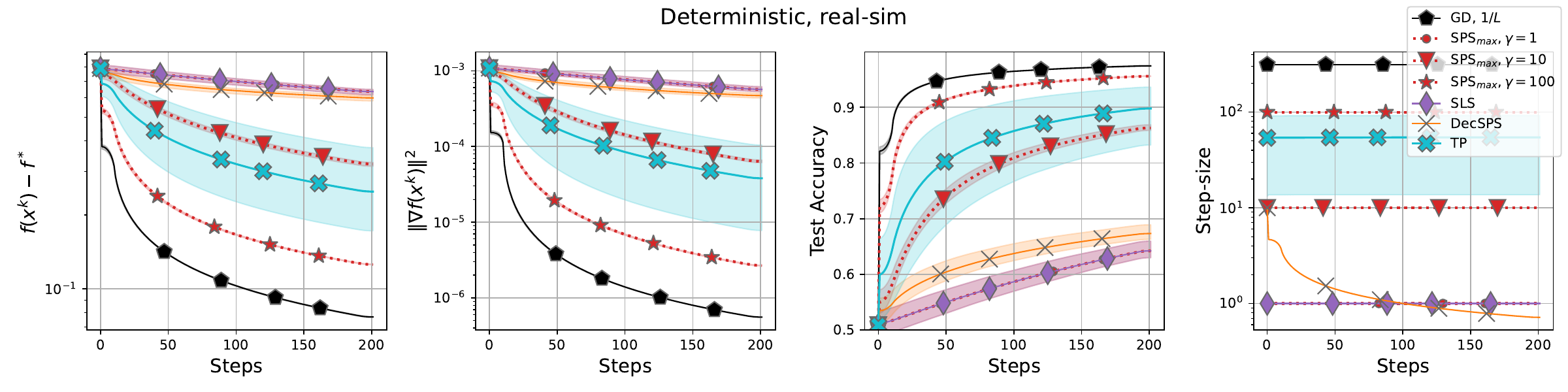}
    \end{subfigure}
    \caption{Performance and stepsize evolution of \tp{} compared to other methods on binary classification problems minimizing \textbf{logistic regression} loss function.}
    \label{exp:determ-logreg-4}
\end{figure*}

\begin{figure*}
    \centering
    \begin{subfigure}{\textwidth}
        \centering
        \includegraphics[width=\linewidth]{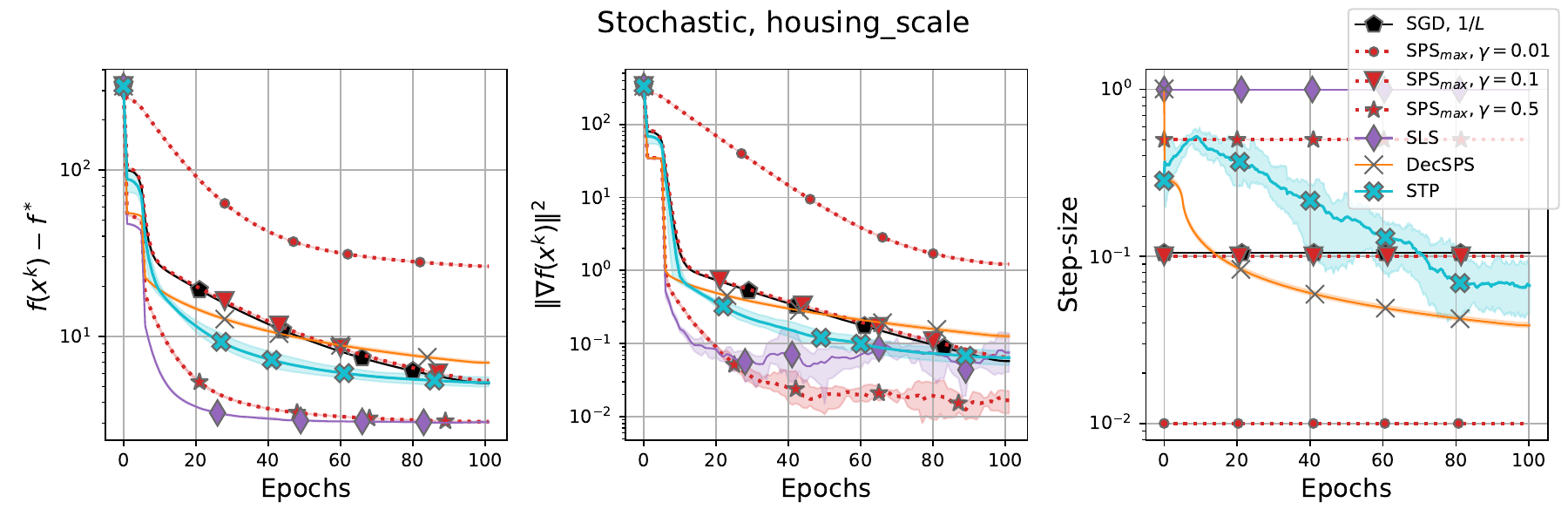}
        
        \includegraphics[width=\linewidth]{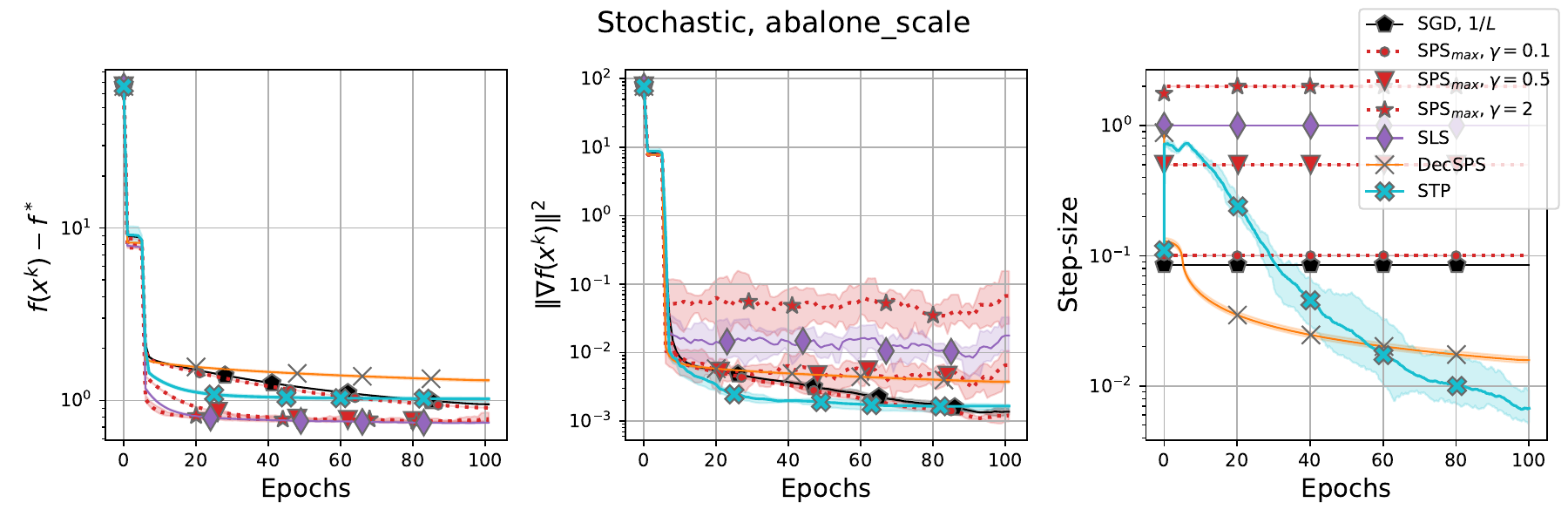}
        
        \includegraphics[width=\linewidth]{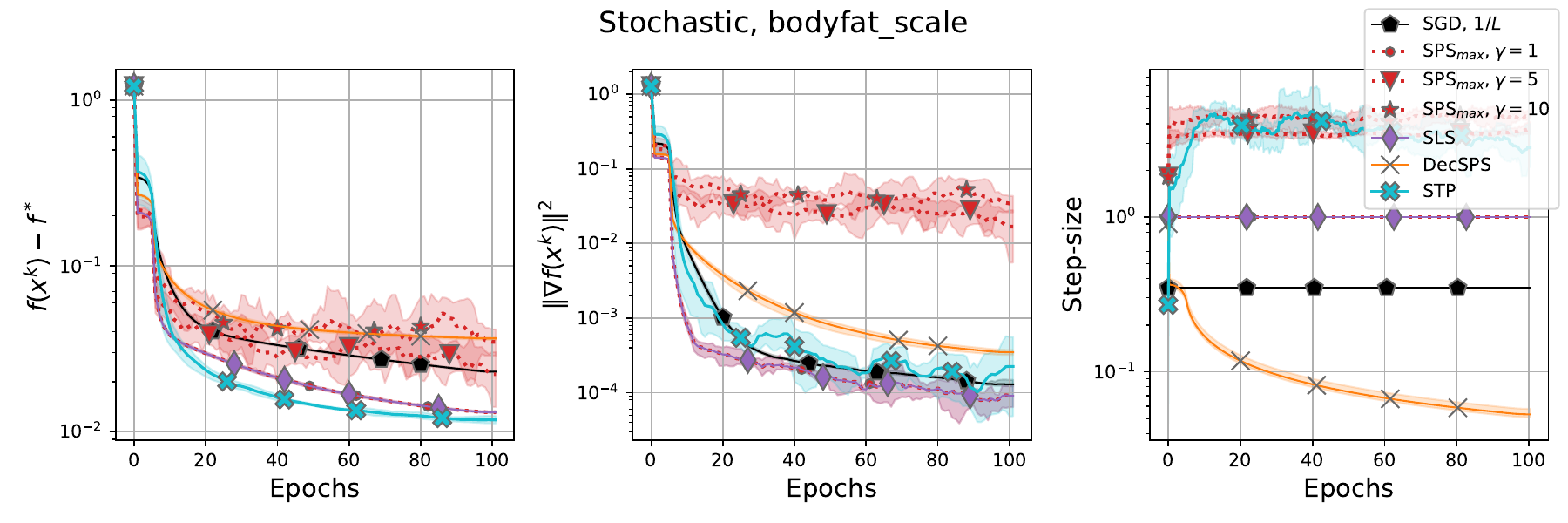}
        
        \includegraphics[width=\linewidth]{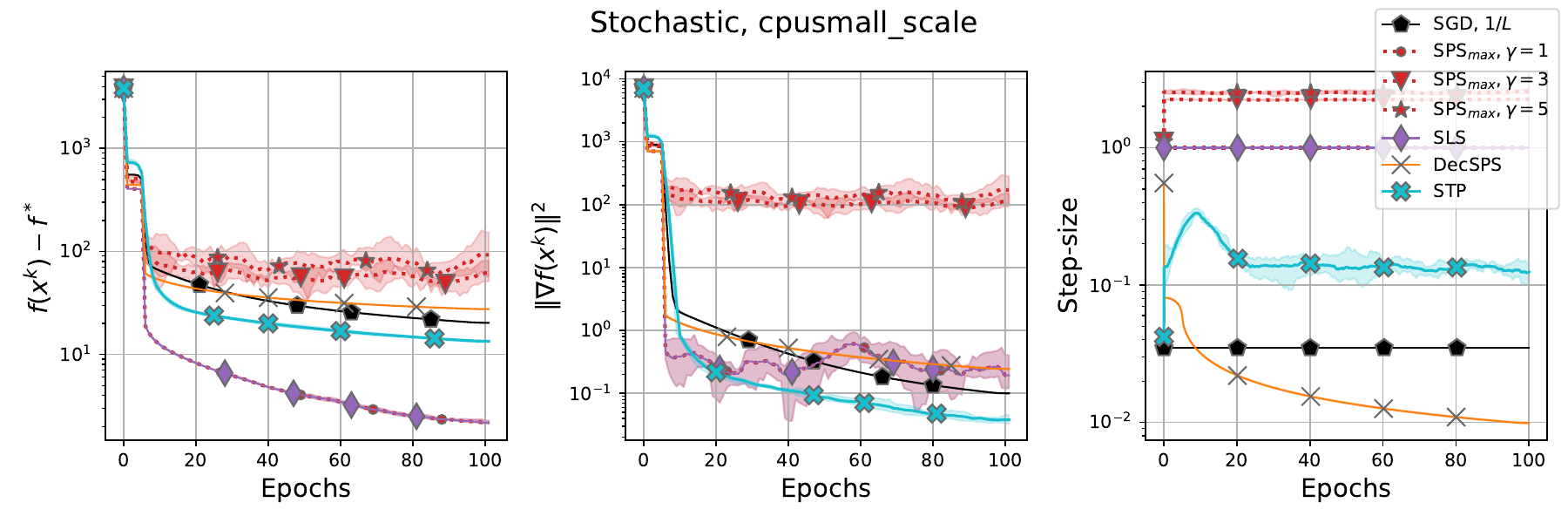}
    \end{subfigure}
    \caption{Performance and stepsize evolution of \stp{} compared to other methods on regression problems minimizing \textbf{least squares} loss function.}
    \label{exp:stoch-lstsq-4}
\end{figure*}

\begin{figure*}
    \centering
    \begin{subfigure}{\textwidth}
        \centering
        \includegraphics[width=\linewidth]{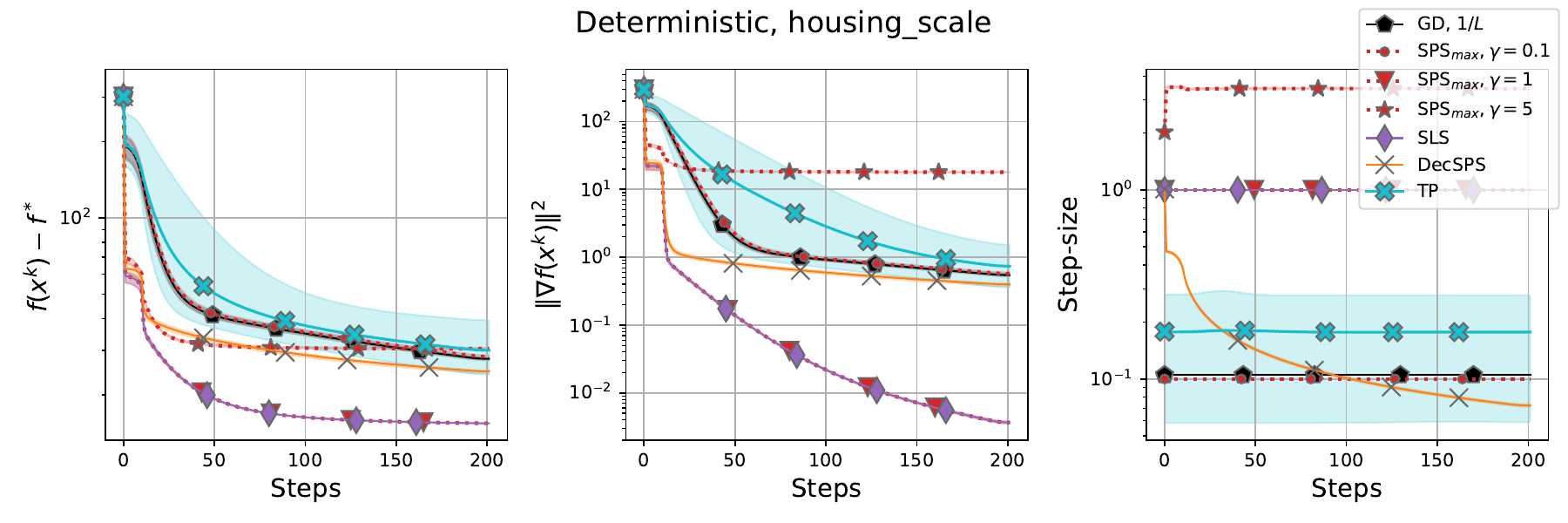}
        
        \includegraphics[width=\linewidth]{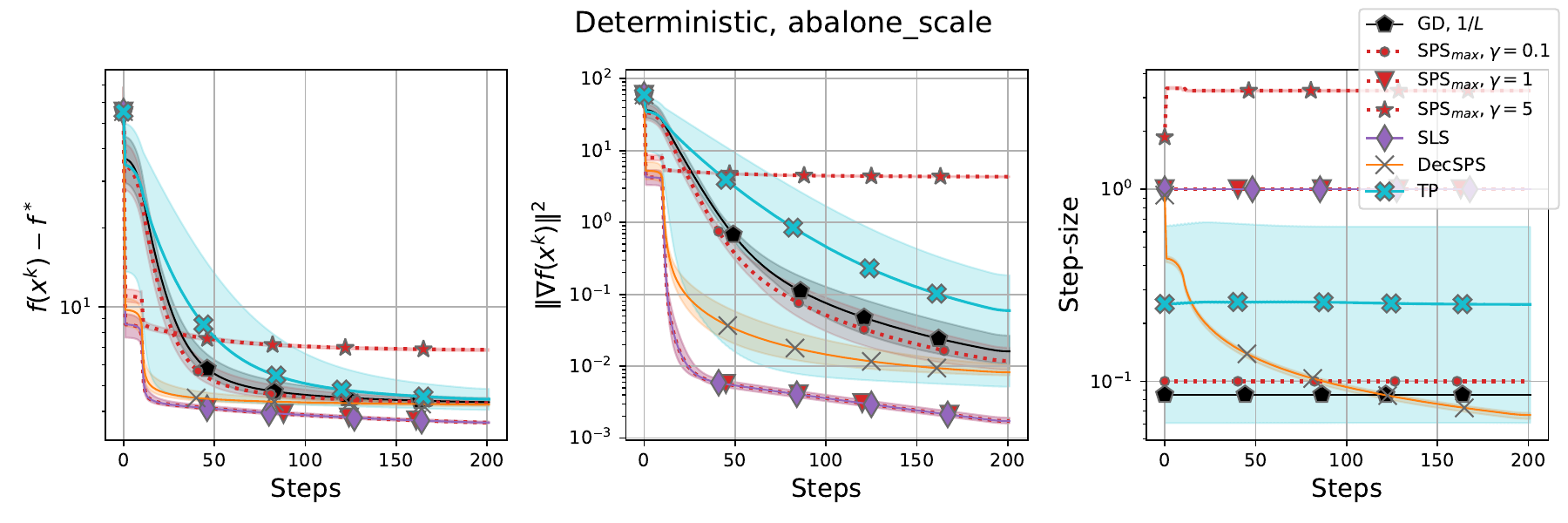}
        
        \includegraphics[width=\linewidth]{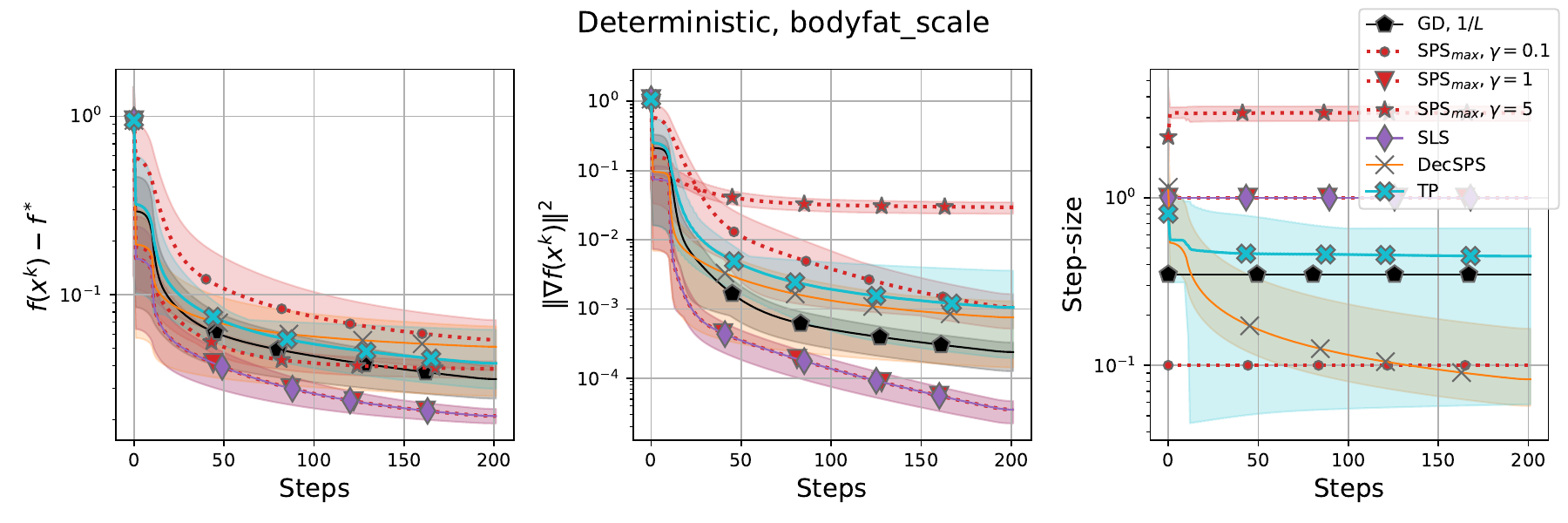}
        
        \includegraphics[width=\linewidth]{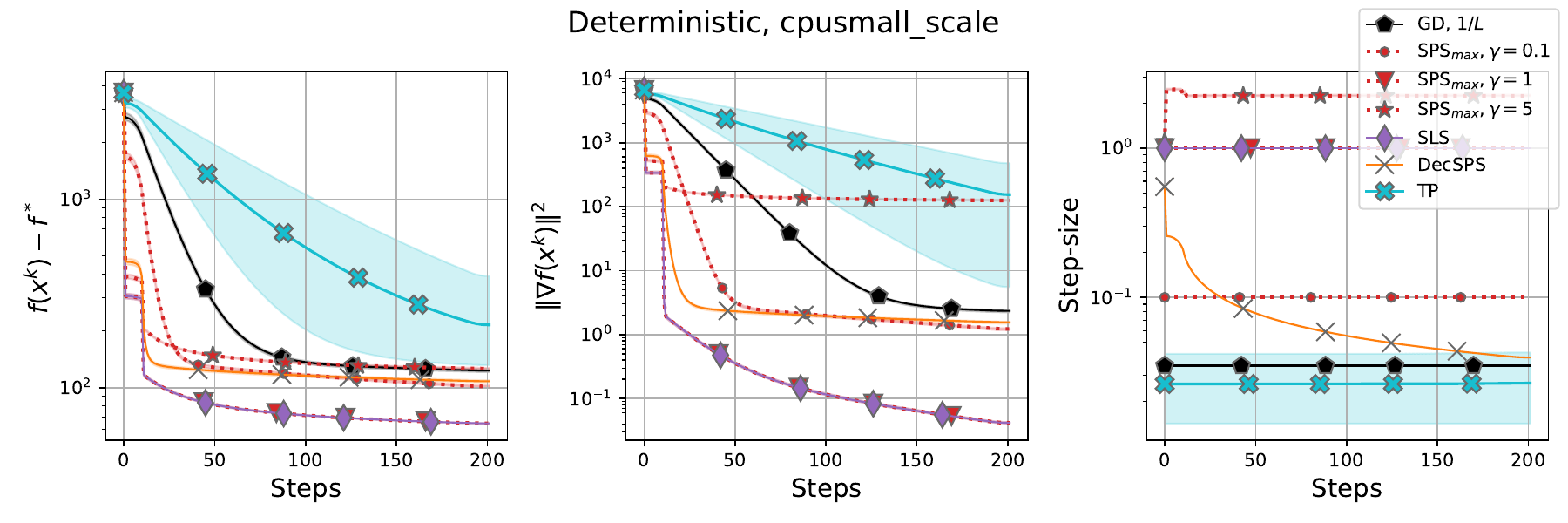}
    \end{subfigure}
    \caption{Performance and stepsize evolution of \tp{} compared to other methods on regression problems minimizing \textbf{least squares} loss function.}
    \label{exp:determ-lstsq-4}
\end{figure*}

\clearpage

\begin{figure*}
    \centering
    \begin{subfigure}{\textwidth}
        \centering
        \includegraphics[width=\linewidth]{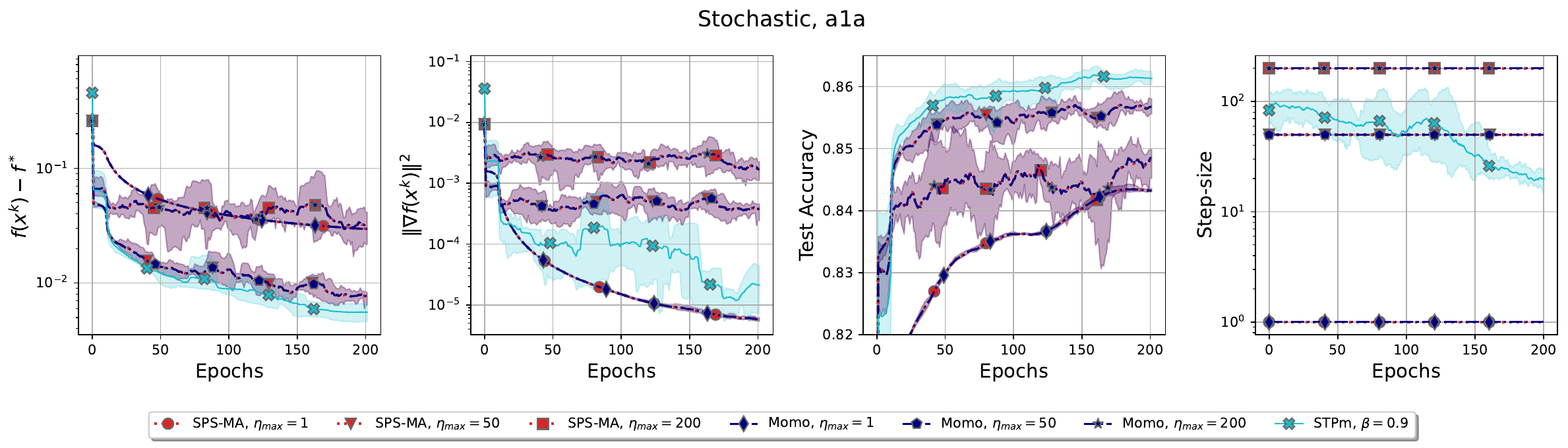}
        
        \includegraphics[width=\linewidth]{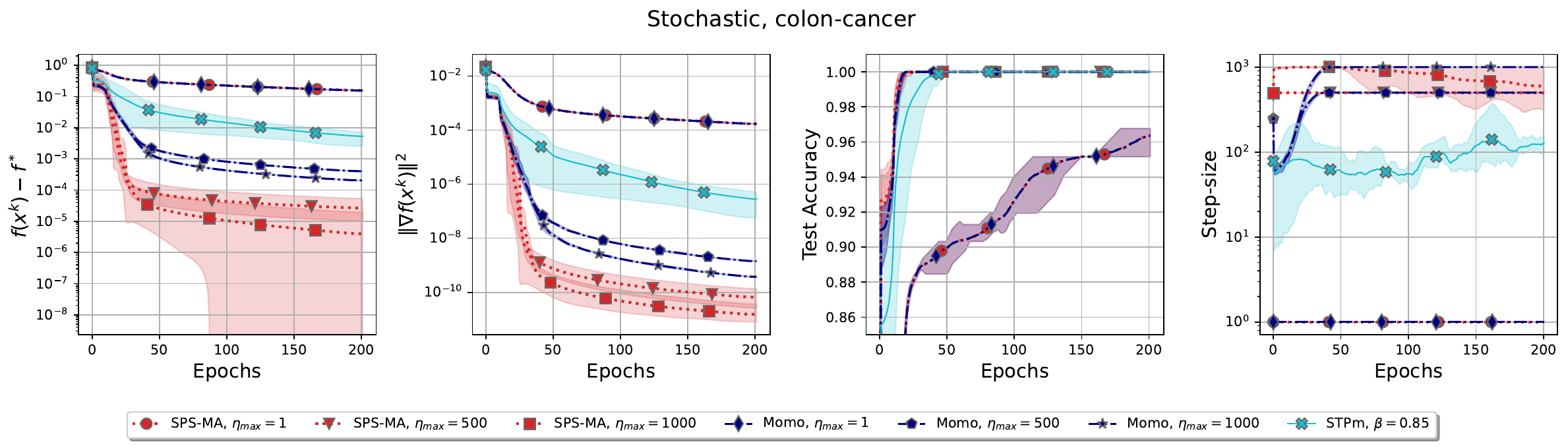}
        
        \includegraphics[width=\linewidth]{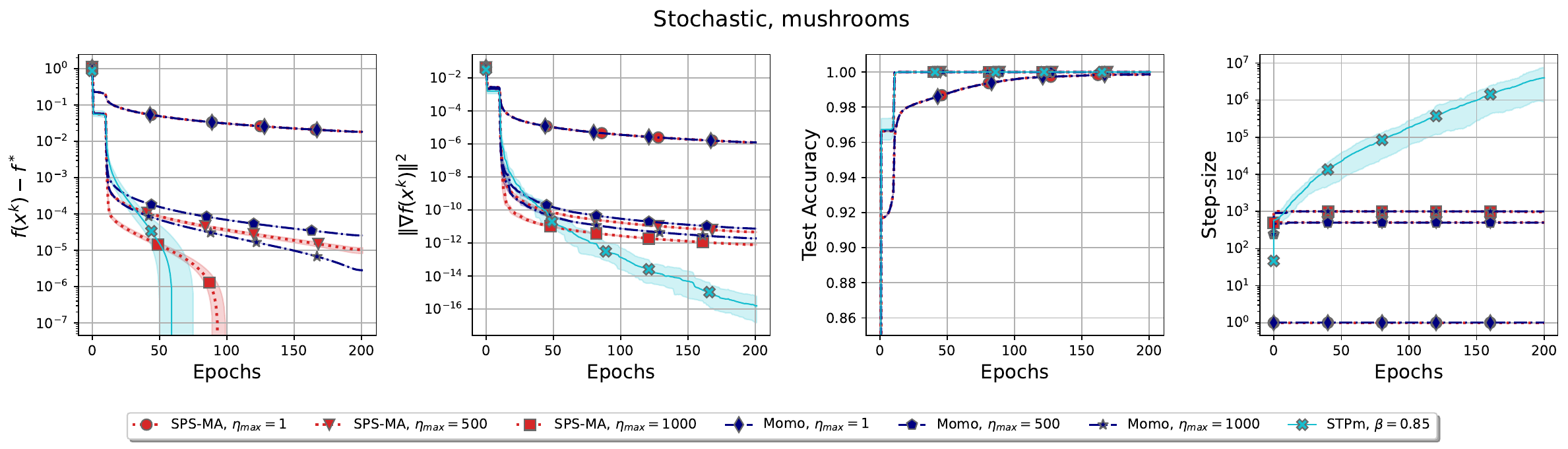}
        
        \includegraphics[width=\linewidth]{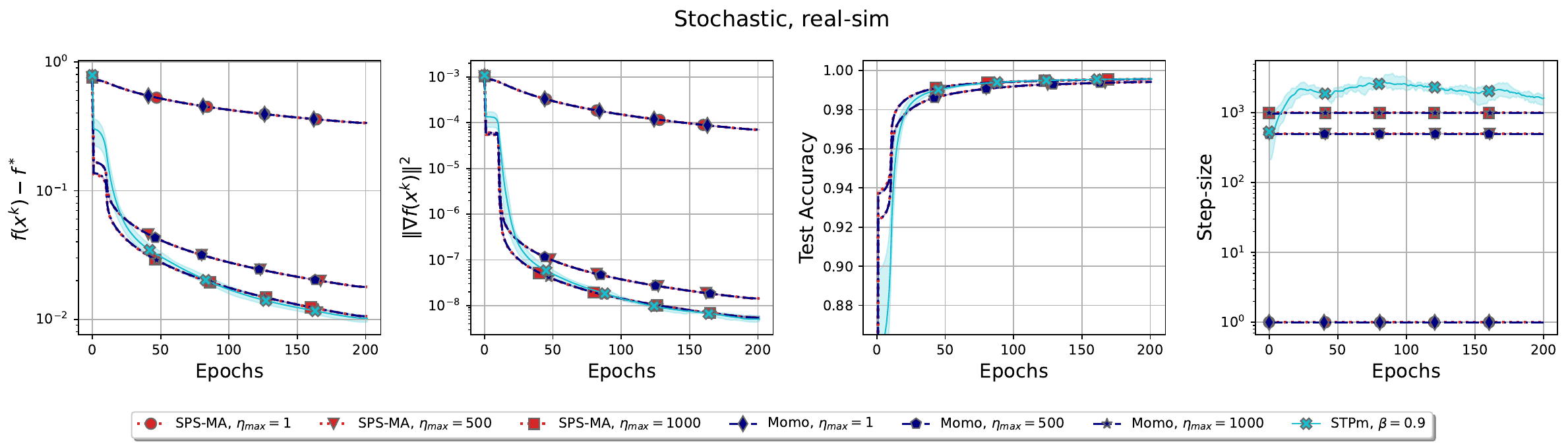}

        \includegraphics[width=\linewidth]{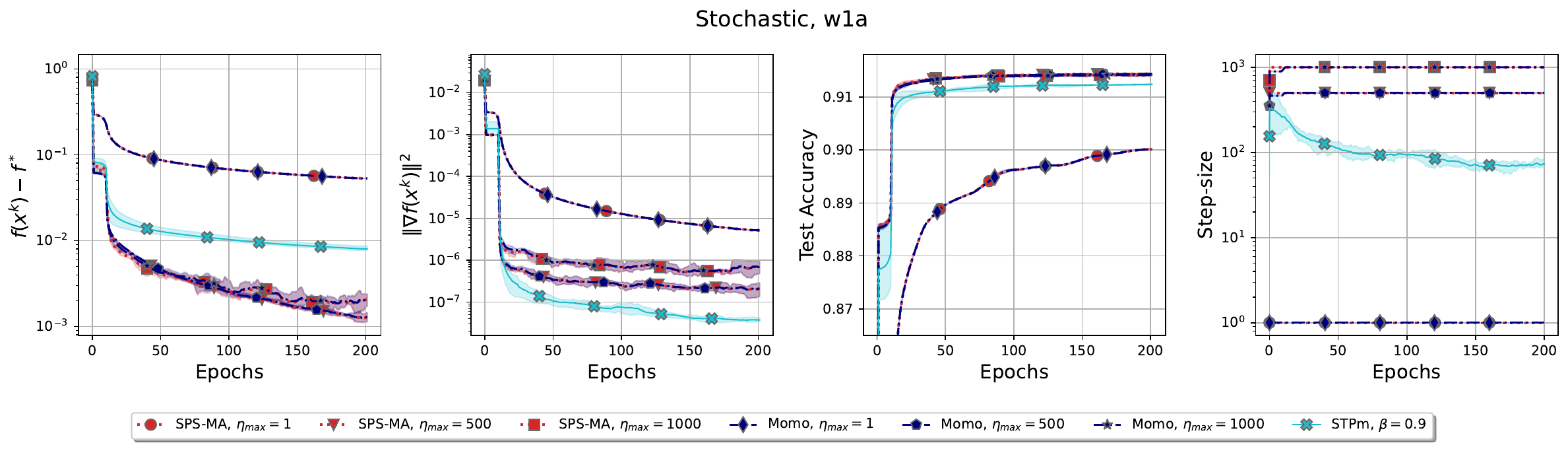}

    \end{subfigure}
    \caption{Performance and stepsize evolution of \stpm{} compared to other methods with momentum on logistic regression problems.}
    \label{exp:stmp}
\end{figure*}

\clearpage
\section{Challenges of the stochastic setup}
\begin{lemma}\label{le:sig}
    Denote differentiable functions $f_i: \rdtor$ for $i \in \lc 1, \dots, n\rc$, their average $f\eqdef \avein in f_i,$ and an arbitrary point $x\in \R^d$.
    For any random variables $s_{i, \zeta}: \mathcal H \to \R$ corresponding to stepsize schedulers, which are based on an independent source of randomness, $\zeta \indep i$, and whole history of optimization algorithm $\mathcal H$\footnote{Informally, history includes all points, functional values, gradients, and all samples of random variables.}.
    If gradient estimator $s_{i,\zeta} \g_i(x)$ is unbiased for any set of functions $f_i$, 
    \begin{align}
        \Edist{i,\zeta}{s_{i,\zeta} \g_i(x)} = \g(x), \label{eq:sig}
    \end{align}
    then necessarily $s_{i,\zeta}$ an $\g_i(x)$ are statistically independent, 
    \begin{equation}
    \Edist {i,\zeta}{s_{i,\zeta} \g_i(x)}= \Edist{i,\zeta}{s_{i, \zeta}}\Edist i{\g_i(x)}.  \label{eq:stat_ind}      
    \end{equation}
\end{lemma}
\begin{proof}[\pof{\Cref{le:sig}}]
    We are going to show that for a given set of functions $f_i$, the solution is unique.
    For $v\in \R^d$ denote $v^i$ its $i$-th coordinate. \Cref{eq:sig} implies that for any differentiable $f_i$ holds
    \begin{align}
        \frac 1n
        \begin{pmatrix}
            \g_1^1(x) & \dots & \g_n^1(x)\\
            \vdots& \ddots & \vdots\\
            \g_1^d(x) & \dots & \g_n^d(x)\\
        \end{pmatrix}
        \begin{pmatrix}
            \Edist{\zeta}{s_{1,\zeta}}\\
            \vdots\\
            \Edist{\zeta}{s_{n,\zeta}}
        \end{pmatrix}
        =
        \begin{pmatrix}
            \g^1(x)\\
            \vdots\\
            \g^d(x)
        \end{pmatrix}.
    \end{align}
    Now consider functions $f_i$ with $d=n$ such that the matrix
    \begin{align}
        \begin{pmatrix}
            \g_1^1(x) & \dots & \g_n^1(x)\\
            \vdots& \ddots & \vdots\\
            \g_1^d(x) & \dots & \g_n^d(x)\\
        \end{pmatrix}
    \end{align}
    is invertible (e.g., $f_i(x)= \frac 12 (x^i)^2$). We have that $\Edist{\zeta} {s_{i,\zeta}}$ can be \textbf{uniquely} expressed as
    \begin{align}        
        \begin{pmatrix}
            \Edist{\zeta}{s_{1,\zeta}}\\
            \vdots\\
            \Edist{\zeta}{s_{n,\zeta}}
        \end{pmatrix}
        =
        \frac 1n
        \begin{pmatrix}
            \g_1^1(x) & \dots & \g_n^1(x)\\
            \vdots& \ddots & \vdots\\
            \g_1^d(x) & \dots & \g_n^d(x)\\
        \end{pmatrix}^{-1}
        \begin{pmatrix}
            \g^1(x)\\
            \vdots\\
            \g^d(x)
        \end{pmatrix}.
    \end{align}
    To conclude the proof, note that the solution of \eqref{eq:sig} are statistically independent $s_{i,\zeta}$ and $\g_i(x)$ from \eqref{eq:stat_ind} with mean one, $\Edist{i,\zeta}{s_{i,\zeta}} =1$.
\end{proof}
\section{Proofs}


\subsection{Proof of Lemma \ref{le:convex_gbound}}
\begin{proof}[\pof{\Cref{le:convex_gbound}}]
    \begin{align}
        \norms {x^{k+1} - \xopt} 
        & = \norms{x^k-\xopt} -2\ssize_k\la \xerr k, \g(x^k)\ra + \ssize_k^2 \norms {\g(x^k)},\\
        \intertext{from convexity}
        & \leq \norms{x^k-\xopt} -2\ssize_k \ls f(x^k)-\fopt \rs + \ssize_k^2 \norms {\g(x^k)},\\
        \intertext{and using stepsize $\ssize_k=2\frac {f(x^k)-f(y^k)}{\norms{\g(x^k)}}$ we get}
        & = \norms{x^k-\xopt} + \lr -4\frac {f(x^k)-f(y^k)}{f(x^k)-\fopt} + 4\frac {\ls f(x^k)-f(y^k) \rs^2}{\ls f(x^k)-\fopt \rs^2} \rr \frac { \ls f(x^k)-\fopt \rs^2}{\norms {\g(x^k)}}\\
        & = \norms{x^k-\xopt} - 4\frac {f(x^k)-f(y^k)}{f(x^k)-\fopt} \lr \frac {f(y^k)-\fopt}{ f(x^k)-\fopt } \rr \frac { \ls f(x^k)-\fopt \rs^2}{\norms {\g(x^k)}},\\
        \intertext{using Assumption \ref{as:xydiff},}
        & \leq \norms{x^k-\xopt} - 4\xydiff \yodiff \frac { \ls f(x^k)-\fopt \rs^2}{\norms {\g(x^k)}}.
    \end{align}
    Now telescoping the sum and using \Cref{def:gbound} we get
    \begin{align}
        \frac {4\xydiff\yodiff} {G^2}\sumin tk  \ls f(x^t)-\fopt \rs^2
        \leq \norms{x^0-\xopt} - \norms{x^{k+1}-\xopt}
        \leq \norms{x^0-\xopt},
    \end{align}
    therefore
    \begin{align}
        \min_{t\in\{1, \dots, k\}}  \ls f(x^t)-\fopt \rs^2
        &\leq \frac {G^2\norms{x^0-\xopt}}{4\xydiff\yodiff k },\\
        \min_{t\in\{1, \dots, k\}}  f(x^t)-\fopt
        &\leq \frac {G\norm{x^0-\xopt}}{2 \sqrt{\xydiff\yodiff k} }.
    \end{align}
\end{proof}

\subsection{Proof of Lemma \ref{le:quadratic}}
\begin{proof}[\pof{\Cref{le:quadratic}}]
    With
    \begin{align}
        \g(x)= x,
        \qquad \normsM{x^k} 2 > \normsM{y^k} 2,
    \end{align}
    we have
    \begin{align}
        x^{k+1}
        &= x^k - \frac {\normsM{x^k}2- \normsM{y^k}2}{\normsM {x^k}2} x^k
        = x^k \frac {\normsM{y^k}2}{\normsM {x^k}2},
    \end{align}
    and we can conclude
    \begin{align}
        f(x^{k+1})
        &= \frac 12 \normsM {x^{k+1}}2 + b
        = \frac 12 \normsM {x^k}2 \lr \frac {\normsM{y^k}2}{\normsM {x^k}2} \rr^2 + b \nonumber\\
        &< \frac 12 \normsM {y^k}2 + b
        = f(y^{k}).
    \end{align}
\end{proof}

\subsection{Proof of Proposition \ref{pr:strongly_convex}}

\begin{proof}[\pof{\Cref{pr:strongly_convex}}] \label{sec:proof_pr_strongly_convex}
    \begin{align}
        \norms {x^{k+1} - \xopt} 
        & = \norms{x^k-\xopt} -2\ssize_k\la \xerr k, \g(x^k)\ra + \ssize_k^2 \norms {\g(x^k)},\\
        \intertext{from strong convexity}
        & \leq \lr 1-\mu \ssize_k \rr \norms{x^k-\xopt} -2\ssize_k \ls f(x^k)-\fopt \rs + \ssize_k^2 \norms {\g(x^k)},\\
        & = \lr 1-\mu \ssize_k \rr \norms{x^k-\xopt} -\ssize_k \lr 2\ls f(x^k)-\fopt \rs - \ssize_k \norms {\g(x^k)} \rr,\\
        \intertext{and using that stepsize stepsize $\ssize_k \leq 2\frac {f(x^k)-\fopt}{\norms{\g(x^k)}}$ we get}
        & \leq \lr 1-\mu \ssize_k \rr \norms{x^k-\xopt}.
    \end{align}
    Chaining the inequality concludes the proof.
\end{proof}

\end{document}